\documentclass[12pt]{article}
\usepackage{amsmath} % assumes amsmath package installed
\usepackage{amsfonts,mathrsfs}
\usepackage{amssymb}
\usepackage{color}
\usepackage{graphicx}
\usepackage{epsfig}

\usepackage{caption,subcaption}
{\bf}{}
{\bf}{}
\newtheorem{corollary}{Corollary}{\bf}{}
{\bf}{}
\newtheorem{lemma}{Lemma}{\bf}{}
\newtheorem{theorem}{Theorem}{\bf}{}
{\bf}{}
\newtheorem{assumption}{Assumption}{\bf}{}
\newenvironment{proof}[1][Proof]{\noindent\textbf{#1.} }{\ \rule{0.5em}{0.5em}}

\def\a{\alpha}

\def\tg{\tilde g}
\def\re{\mathbb{R}}

\def\F{\mathcal{F}}
\def\argmin{\mathop{\rm argmin}}
\newcommand{\EXP}[1]{\mathsf{E}\!\left[#1\right] }

\title{On Stochastic Subgradient Mirror-Descent Algorithm with Weighted Averaging}
\author{Angelia Nedi\'c and Soomin Lee}
\begin{document}
\maketitle

\centerline{Dedicated to Paul Tseng}
\begin{abstract}
This paper considers stochastic subgradient mirror-descent method
%, in combination with simple iterate averaging,
for solving constrained convex  minimization problems.
In particular, a stochastic subgradient mirror-descent method
with weighted iterate-averaging is investigated and its per-iterate convergence rate is analyzed.
The novel part of the approach is in the choice of weights that are used to construct the averages.
Through the use of  these weighted averages, we show that the known
optimal rates can be obtained with simpler algorithms than those currently existing in the literature.
Specifically, by suitably choosing the stepsize values, one can obtain the rate of the order
$1/k$ for strongly convex
functions, and the rate $1/\sqrt{k}$ for general convex functions (not necessarily differentiable).
Furthermore, for the latter case, it is shown that a stochastic subgradient mirror-descent
with iterate averaging
converges (along a subsequence) to an optimal solution, almost surely,
even with the stepsize of the form
$1/\sqrt{1+k}$, which was not previously known. The stepsize choices that achieve the best rates are
those proposed by Paul Tseng for acceleration of proximal gradient methods~\cite{Tseng2008}.
\end{abstract}

%%%%%%%%%%%%%%%%%%%%%%%%%%%%%%%%%%%%%%%%%%%%%%
\section{Introduction}\label{sec:introduction}
The work in this paper is motivated by several recent papers showing that using averaging is beneficial when constructing fast (sub)gradient algorithms for solving convex optimization problems. Specifically, for problems where the objective function has Lipschitz continuous gradients, Tseng~\cite{Tseng2008} has recently proposed an accelerated gradient method that uses averaging to construct a generic algorithm with the convergence rate of $\frac{1}{k^2}$. This convergence rate is known to be the best in the class of convex functions with Lipshitz gradients~\cite{NesterovBook2004}, for which the first fast algorithm is originally constructed by Nesterov~\cite{Nesterov83} for unconstrained problems, and recently extended in~\cite{Beck2009} to a larger class of problems. Averaging has also recently been used by Ghadimi and Lan
in~\cite{GhadimiLan2012} to develop an algorithm that has the rate $\frac{1}{k^2}$ if the objective function has Lipschitz continuous gradients, and the rate $\frac{1}{k}$ if the objective function is strongly convex (even if stochastic subgradient is used). Some interesting results have been shown by
Juditsky et al.~\cite{Juditsky2008} for mirror-descent algorithm
with averaging as employed to construct aggregate estimators with the best achievable learning rate.

Recently, Lan~\cite{Lan2012}
has considered averaging technique for the mirror-descent algorithm
for stochastic composite problems (involving the sum of a smooth objective and a nonsmooth objective function), where the accelerated stepsizes akin to those considered by
Tseng~\cite{Tseng2008} have also been proposed.
The algorithms proposed by Tseng in~\cite{Tseng2008}, and by Ghadimi and Lan
in~\cite{GhadimiLan2012}, rely on a construction of three sequences, some of which use a form of averaging.
A different form of averaging has been considered by Nesterov in~\cite{Nesterov2009},
where the averaging is used in both primal and dual spaces
to construct a subgradient method with the rate $\frac{1}{\sqrt{k}}$,
which is known to be the best convergence rate of any first-order method for convex functions in
general~\cite{NesterovBook2004}. A much simpler iterate averaging scheme dates back to Nemirovski and
Yudin~\cite{NemYudin1978} for convex-concave saddle-point problems. Such a scheme has also
been considered by Polyak and Juditsky~\cite{Polyak1992} for stochastic (gradient) approximations,
and by Polyak~\cite{Polyak01} for convex feasibility problems.
A somewhat different approach has been
considered by Juditsky et al.~\cite{Juditsky2005}, where a variant of the mirror-descent algorithm with averaging has been proposed for the classification problem.
More recently, the simple iterate averaging
has been considered by Nemirovski et al.~\cite{Nemirovski2009}  to show the best achievable rate in
the context of stochastic subgradients. Recently, Juditsky and Nesterov~\cite{JudNesterov2010}
have further investigated some special extensions of the primal-dual averaging method
for a more general class of uniformly convex functions,
while Rakhlin et al.~\cite{Rakhlin2012} have investigated a form of ``truncated averaging'' of the iterates
for a stochastic subgradient method in order to achieve the best known rate for strongly convex functions.

In this paper, we further explore the benefits of averaging by providing a somewhat
different analytical approach to the stochastic subgradient mirror-descent methods.
In particular, we consider a stochastic subgradient mirror-descent method combined with a simple
averaging of the iterates. The averaging process is motivated by that of
Nemirovski and Yudin~\cite{NemYudin1978}, which was also used later on
by Polyak and Juditsky~\cite{Polyak1992} and by Polyak~\cite{Polyak01}.
In this averaging process, the averaged iterates are not used in the construction of the algorithms,
but rather occur as byproducts of the algorithms, where the averaging weights
are specified in terms of the stepsizes that the algorithm is using. The novel part of this work
is in the choice of the stepsize  (and averaging weights),
which are motivated by those proposed by Tseng~\cite{Tseng2008} and include
the Nesterov stepsize~\cite{Nesterov83,NesterovBook2004}.
The development relies on a new choice of ``a Lyapunov function"
that is used to measure the progress of an algorithm, which
combined with a relatively simple analysis allows us
 to recover the known rate results
and also develop some new almost sure convergence results.

Specifically, we consider two cases namely, the case when the objective function is strongly convex
while the constraint set is just convex and closed, and the case when the objective function is
just convex (not necessarily differentiable) while the constraint set is convex and compact.
In both cases, our algorithm achieves the best known convergence rates.
For strongly convex functions, we show that the algorithm with averaging achieves
the best convergence rate of $\frac{1}{k}$ per iteration $k$. This result is the same as that
of~Juditsky et al.~\cite{Juditsky2005},
Ghadimi and Lan in~\cite{GhadimiLan2012}, and Rakhlin et al.~\cite{Rakhlin2012}.
However, we show that this optimal
rate is attained with a simpler algorithm than the algorithms in~\cite{GhadimiLan2012},~\cite{Juditsky2005}, and with
an averaging that is different from the one used in~\cite{Rakhlin2012}.
For a compact constraint set, our algorithm achieves the best known rate of
$\frac{1}{\sqrt{k}}$ at iteration $k$ by using the stepsize of the form $\frac{1}{\sqrt{k}}$.
The rate $\frac{1}{\sqrt{k}}$ is achievable by a time-varying stepsize sequence with an averaging
over all iterates that are generated up to a given time,
which is different from the window-based averaging proposed in~\cite{GhadimiLan2012}.
The novel part of the work is in the establishment of the almost sure sub-sequential convergence for the averaging sequence obtained by the method
with a non-summable stepsize $\alpha_k=\frac{1}{\sqrt{k}}$ (as given in Theorem~4). To the best of our knowledge, this is the first almost sure convergence result for a non-summable stepsize.
The existing convergence results for the stochastic mirror-descent (hence, for the stochastic subgradient method) that uses a non-summable stepsize show
the convergence of the function values in the expectation only~\cite{GhadimiLan2012},~\cite{Juditsky2005}.
Our result is new even for the mirror-descent method
(hence, also for the subgradient method) without stochastic errors, as it also shows the sub-sequential
convergence of the average sequence to an optimal solution.

The paper is organized as follows. In Section~\ref{sec:prelim}, we formalize the problem, describe the
basic stochastic subgradient mirror-descent method and discuss our assumptions.
In Section~\ref{sec:stronglyconvex}, we present the results for the algorithm with
iterate averaging for strongly convex functions. In Section~\ref{sec:compact}, we analyze the convergence properties of the algorithm for the case when the constraint set is compact. We report some simulation results
in Section~\ref{sec:numerics} and provide concluding remarks in Section~\ref{sec:con}.

%Given a closed convex set $C\subset \re^n$, we let $\Pi_C[x]$ denote the projection of a vector $x$ on the set $C$. We will often use the non-expansiveness property of the projection operation:
%\begin{equation}\label{eqn:proj}
%$\|y-\Pi_X[x]\|\le \|y-x\|$ for all $x\in\re^n$ and $y\in X.$ %\end{equation}
%%%%%%%%%%%%%%%%%%%%%%%%%%%%%%%%%%%%%%%%%%%%%%%%%%%%%%%%%%%%%%%%%

\section{Stochastic Subgradient Mirror-Descent Algorithm}\label{sec:prelim}
Consider the problem of minimizing a convex but not necessarily differentiable function $f$
over a constraint set $X$:
\begin{eqnarray}\label{eqn:prob}
\hbox{minimize \ } && f(x)\cr
\hbox{subject to  } && x\in X.\end{eqnarray}
We will use $f^*$ to denote the optimal value of the problem and $X^*$ to denote the solution set of the problem,
\[f^*=\inf_{x\in X} f(x),\qquad X^*=\{x\in X\mid f(x)=\min_{y\in X} f(y)\}.\]

The set $X$ is assumed to be {\it convex and closed}, while the function $f$ is assumed to be {\it convex and continuous at all points $x\in X$}. In addition,
{\it a subgradient $g(x)$ is assumed to exists} at every point $x\in X$, i.e., for every
$x\in X$, there is a vector $g(x)$ such that
\[f(x)+\langle g(x), y-x\rangle \le f(y)\qquad\hbox{for all }y\in X.\]
In principle, a subgradient definition requires that the above inequality is satisfied for all $x$ in the domain of $f$, but for the purpose of our discussion it suffices to have the inequality valid just over the set $X$.
All the aforementioned assumptions are blanket assumptions for the rest of the paper.
Regarding the notation used throughout the paper, let us note that  we view
vectors as column vectors,  and
we use $\langle x,y\rangle$ to denote the inner product of two vectors $x,y\in\re^n$.
We assume that $\re^n$ is equipped with some norm $\|\cdot\|$, and use $\|\cdot\|_*$ to denote its
dual norm.

Strongly convex non-differentiable optimization problems arise most prominently in machine learning
as regularized stochastic learning problems~\cite{Vapnik}, (also for example, see~\cite{Xiao2010} and the detailed literature overview therein). These problems are of the following generic form:
\[\hbox{minimize } \EXP{f(x,u)} + g(x)\qquad \hbox{over all $x\in\re^n$},\]
where $x$ is the optimization variable and $u$ is a random vector  with an unknown distribution governing
the random observations of the input-output data pairs. The function $f(x,u)$ is a loss function, which is convex but often non-differentiable, while $g(x)$ is a strongly convex function that regularizes the problem by promoting some desired features for the optimal solution of the problem.
As a specific example, consider
the maximum-margin separating hyperplane problem,
one of the canonical classification problems within the support vector machines methodology, which can be described as follows.
%\begin{example}
Given a set of $M$ data-label pairs $\{(a_j; b_j),\, 1\le j\le M\}$, where $a_j\in\re^n$
and $b_j\in\{-1,+1\}$ for all $j$, we want to find a vector $x^*\in\re^n$ that solves
the following convex optimization problem:
\[\hbox{minimize } f(x)
=\frac{\lambda}{2}\|x\|^2 + \frac{1}{M}\sum_{i=1}^M\max\{b_j\langle a_j,x\rangle-1,\, 0\}
\quad \hbox{over all $x\in\re^n$},\]
where $\lambda>0$ is a regularization parameter (also a strong convexity constant for the objective function).
%The problem is non-smooth and has bounded subgradients.
The second term in the objective function is the
empirical estimate of the expected loss based on $M$ random observations of input-output data pairs.
The optimal solution to this problem is known as the maximum-margin separating hyperplane~\cite{Vapnik}.
%\end{example}

We consider stochastic subgradient mirror-descent algorithm for solving problem~\eqref{eqn:prob}.
In particular, we assume that instead of a subgradient $g(x)$ at a point $x$, we can compute an erroneous
subgradient $\tg(x)$ and use it within the mirror-descent algorithm. The mirror-descent algorithm, as proposed in~\cite{Yudin1978}, is a generalization of the standard subgradient method where the Euclidean norm is replaced with a generic Bregman distance function~\cite{Bregman1967}.
The Bregman distance function is defined in terms of a continuously differentiable
and strongly convex function
$w(\cdot)$ over the set $X$, with a scalar $\mu_w>0$, which satisfies
\[w(y)\ge w(x)+\langle\nabla w(x),y-x\rangle +\frac{\mu_w}{2}\|y-x\|^2\qquad\hbox{for all $x,y\in X$}.\]
The Bregman distance function induced by $w(\cdot)$ is denoted by $D_w$ and given by
\[D_w(x,z)=w(z) - w(x) -\langle\nabla w(x),z-x\rangle \qquad\hbox{for all $x,z\in X$}.\]
From the definition it can be seen that the Bregman distance function has the following properties
\begin{equation}\label{eqn:bdprop}
D_w(x,z)-D_w(y,z)=D_w(x,y)+\langle\nabla w(y)-\nabla w(x),z-y\rangle\qquad\hbox{for all $x,y,z\in X$},
\end{equation}
\begin{equation}\label{eqn:bdstr}
D_w(x,z)\ge \frac{\mu_w }{2}\|x-z\|^2\qquad\hbox{for all $x,z\in X$},
\end{equation}
where relation~\eqref{eqn:bdstr} follows by the strong convexity of the function $w$.
Furthermore, $D_w(x,z)$ is differentiable with respect to $z$. Letting
$\nabla_z D_w(\cdot,\cdot)$ denote the partial derivative of $D_w(x,z)$ with respect to the $z$ variable,
we have
\begin{equation}\label{eqn:gradz}
\nabla_z D_w(x,z)=\nabla w(z) -\nabla w(x)\qquad\hbox{for all $x,z\in X$}.
\end{equation}

A subgradient mirror-descent method generates iterates, starting with an initial point $x_0\in X$,
according to the following update rule:
\begin{equation*}
x_{k+1}=\argmin_{z\in X}\left\{\a_k\langle g_k,z-x_k\rangle + D_w(x_k,z)\right\}
\qquad\hbox{for all }k\ge0,\end{equation*}
where $\a_k>0$ is a stepsize and $g_k$ is a subgradient of $f(x)$ evaluated at $x=x_k$.
The algorithm works under the premise that the set $X$ has a structure admitting efficient
computation of $x_{k+1}$, such as for example when a closed form of $x_{k+1}$ is available.

In order to deal with a more general class of methods, we will assume that subgradients are evaluated with some random error, and these erroneous subgradients are used instead of the subgradients, i.e.,
the subgradient $g_k$ is replaced by a  noisy subgradient $\tilde g_k$.
This gives rise to a stochastic mirror-descent algorithm of the following form:
\begin{equation}\label{eqn:mdstoch}
x_{k+1}=\argmin_{z\in X}\left\{\a_k\langle \tilde g_k,z-x_k\rangle +D_w(x_k,z)\right\}
\qquad\hbox{for all }k\ge0,
\end{equation}
where the initial point $x_0\in X$ may also be random with $\EXP{\|x_0\|^2}<\infty$, but independent
of the random subgradient process $\{\tg_k\}$.
This algorithm also arises when the objective function is given as the expectation of a random function, i.e.,
$f(x)=\EXP{F(x,\xi)}$. In this case, at iteration $k$, a sample function $F(x_k,\xi_k)$
is assumed to be available and $\tg_k$ is a subgradient of $F(x,\xi_k)$ at $x=x_k$.

The standard stochastic subgradient method is a special case of method~\eqref{eqn:mdstoch},
when $w(x)=\frac{1}{2}\|x\|_2,$ where $\|\cdot\|_2$ is the Euclidean distance. In this case,
the Bregman distance function reduces to
\[D_w(x,z)=\frac{1}{2}\|x-z\|_2^2,\]
and the stochastic mirror-descent method becomes the standard stochastic subgradient-projection method:
\[x_{k+1}=\argmin_{y\in X}\left\{\a_k\langle
\tg_k,y-x_k\rangle+\frac{1}{2}\|y-x_k\|_2^2\right\}\qquad\hbox{for all }k\ge0.\]
%\begin{equation*}%\label{eqn:algo}
%x_{k+1}=\Pi_X\left[ x_k - \g_{k} \tg_k \right]\qquad\hbox{for all }k\ge0,
%\end{equation*}
%where $\Pi_X[x]$ denotes the Euclidean projection on the set $X$.
%where $g_k$ is a subgradient of f(x) evaluated at the point $x=x_k$, the vector $x_0\in X$
%is an initial point and $\a_k>0$ is a step size.

In addition to the iterate sequence $\{x_k\}$ generated by the stochastic mirror-descent algorithm,
we will also consider a sequence $\{\hat x_k\}$ of weighted-averages of the iterates,
with $\hat x_k$ defined by
\[\hat x_k=\sum_{t=0}^k\beta_t x_t,\]
where $\beta_0,\beta_1,\ldots, \beta_k$ are non-negative scalars with the sum equal to 1. These convex weights will be appropriately defined in terms of the stepsize values $\a_0,\a_1,\ldots,\a_k.$
The weight selection (and the forthcoming analysis) are motivated by an alternative re-scaled version
of the mirror-descent algorithm:
\begin{equation}\label{eqn:alter}
x_{k+1}=\argmin_{z\in X}\left\{\langle \tilde g_k,z-x_k\rangle +\frac{1}{\a_k}D_w(x_k,z)\right\},
\end{equation}
where the scalar $\frac{1}{\a_k}$ is interpreted as a penalty value at points $z\in X$ that are far from
$x_k$ in terms of the Bregman distance. Intuitively, as the method progresses and we have
a higher confidence in the quality of the iterate $x_k$, it makes sense to increase the penalty
$\frac{1}{\a_k}$ for deviating from $x_k$ too far, which is done by decreasing the stepsize $\a_k$ to 0.
However, the decrease rate of the stepsize $\a_k$ is crucial for the convergence rate of the algorithm,
and this rate has to be adjusted depending on the properties of the objective function $f$.

The alternative description~\eqref{eqn:alter} of the stochastic mirror-descent algorithm
suggests that we may use the weighted (expected) Bregman distance function
$\frac{1}{\a_k}\EXP{D_w(x_k,y)}$, with $y\in X$, as a Lyapunov function to measure the progress of the method, which may provide us with some additional insights about the convergence of the method.
This point of view motivates our choice of the weighted iterate-averages and the overall
development in the rest of the paper.

In what follows, we assume that the stochastic subgradients are well behaved in the sense of the following
assumption.
\begin{assumption}\label{asum:stochmir}
Let the stochastic subgradient $\tilde g(x)$ be such that almost surely
\[\EXP{\tilde g(x)\mid x}=g(x)\qquad\hbox{for all }x\in X.\]
\[\EXP{\|\tilde g(x)\|_*^2\mid x}\le \tilde C^2\qquad
\hbox{for some scalar $\tilde C>0$ and for all $x\in X$}.\]
\end{assumption}
\noindent
When subgradients are evaluated without any error, i.e., $\tilde g(x)=g(x)$,  Assumption~\ref{asum:stochmir}
is satisfied if the subgradients are uniformly bounded over $X$, i.e.,
$\|g(x)\|\le C$ for all $x\in X$ and some $C>0$, such as for example when $X$ is compact.
Furthermore, if the subgradients are uniformly
bounded to start with, and the subgradient errors are such that $\EXP{\tilde g(x)\mid x}=g(x)$ and
$\EXP{\|\tilde g(x)-g(x)\|_*^2\mid x}\le \nu^2$ for some $\nu>0$ and all $x\in X$ almost surely, then
Assumption~\ref{asum:stochmir} holds with $\tilde C^2=2C^2 +2\nu^2$ for a general norm.
Moreover, if the norm $\|\cdot\|$ is the Euclidean norm, then
Assumption~\ref{asum:stochmir} holds with %be replaced with
$\tilde C^2=C^2 +\nu^2$. %for the Euclidean norm.

Now, we define the $\sigma$-field generated by the history of the algorithm:
\[\F_k=\sigma\{x_0,\tg_0,\ldots,\tg_{k-1}\}\qquad\hbox{for }k\ge1,\]
with $\F_0=\sigma\{x_0\}$. Using this $\sigma$-field,
%\begin{assumption}\label{asum:stochmir}
%Let the stochastic subgradient process $\{\tilde g_k\}$ be such that almost surely
%\[\EXP{\tilde g_k\mid \F_k}=g_k\qquad\hbox{for all }k\ge0.\]
%\[\EXP{\|\tilde g_k\|_*^2\mid \F_k}\le \tilde C^2\qquad
%\hbox{for all $k\ge0$ and for some scalar $\tilde C>0$}.\]
%\end{assumption}
under  Assumption~\ref{asum:stochmir}, we have the following basic property of the iterates
$\{x_k\}$ generated by the stochastic mirror-descent algorithm.

\begin{lemma}\label{lemma:basic}
Let Assumption~\ref{asum:stochmir} hold.
Then, for method~\eqref{eqn:mdstoch} we have almost surely for all $z\in X$ and $k\ge0$,
\[\EXP{D_w(x_{k+1},z)\mid \F_k} + \a_k \langle g_k, x_k-z\rangle
\le D_w(x_k,z)+\frac{\a_k^2\tilde C^2}{2\mu_w}.\]
\end{lemma}

\begin{proof}
By the first-order optimality condition for the point $x_{k+1}$, we have
\[0\le \left\langle \a_k\tilde g_k +\nabla_z D_w(x_k,x_{k+1}), z-x_{k+1}\right\rangle
\qquad\hbox{for all }z\in X,\]
where $\nabla_z D_w(\cdot,\cdot)$ denotes the partial derivative of the Bregman distance function with respect to the second variable.
Using $\nabla_z D_w(x,z)=\nabla w(z)-\nabla w(x)$ (cf.~Eq.~\eqref{eqn:gradz}), we have
\[0\le \langle\a_k \tilde g_k+\nabla w(x_{k+1})-\nabla w(x_k), z-x_{k+1}\rangle
\qquad\hbox{for all }z\in X,\]
or equivalently
\begin{equation}\label{eqn:optone-err}
\a_k \langle \tilde g_k,x_{k+1}-z\rangle \le \langle\nabla w(x_{k+1})-\nabla w(x_k),z-x_{k+1}\rangle.
\end{equation}

From relation~\eqref{eqn:bdprop}, with $x=x_k$, $y=x_{k+1}$ and an arbitrary $z\in X$, we obtain
\[\left\langle \nabla w(x_{k+1})-\nabla w(x_k),z-x_{k+1}\right\rangle
= D_w(x_k,z)-D_w(x_{k+1},z)-D_w(x_k,x_{k+1}).\]
Substituting the preceding equality in Eq.~\eqref{eqn:optone-err}, we see that for all $z\in X$,
\begin{equation*}
\a_k \langle \tilde g_k,x_{k+1}-z\rangle \le D_w(x_k,z)-D_w(x_{k+1},z)-D_w(x_k,x_{k+1}).
\end{equation*}
By the strong convexity of $w(x)$, we have $D_w(x_k,x_{k+1})\ge \frac{\mu_w}{2}\, \|x_k-x_{k+1}\|^2$
(cf.~Eq.~\eqref{eqn:bdstr})
implying that
\begin{equation}\label{eqn:relac}
\a_k \langle \tilde g_k,x_{k+1}-z\rangle \le D_w(x_k,z)-D_w(x_{k+1},z)- \frac{\mu_w}{2}\, \|x_k-x_{k+1}\|^2.
\end{equation}
Next we estimate the inner-product term $\langle \tilde g_k,x_{k+1}-z\rangle$, as follows:
\begin{eqnarray}\label{eqn:esto}
\a_k\langle \tilde g_k,x_{k+1}-z\rangle
&=&\a_k\langle \tilde g_k, x_{k+1}-x_k\rangle +\a_k\langle \tilde g_k, x_{k}-z\rangle\cr
&\ge & -\left|\langle \frac{\a_k }{\sqrt{\mu_w}}\tilde g_k, \sqrt{\mu_w} (x_{k+1}-x_k)\rangle\right|
+ \a_k \langle \tilde g_k,x_k-z\rangle \cr
&\ge & -\frac{\a_k^2}{2\mu_w}\|\tilde g_k\|_*^2  - \frac{\mu_w}{2} \|x_{k+1}-x_k\|^2
+ \a_k \langle \tilde g_k, x_k-z\rangle ,\end{eqnarray}
where the last relation follows from Fenchel's inequality, i.e.,
$|\langle p,q\rangle|\le \frac{1}{2}\|p\|^2 +\frac{1}{2}\|q\|_*^2$,
and $\|\cdot\|_*$ is the conjugate norm for $\|\cdot\|$.
Upon substituting~\eqref{eqn:esto} in relation~\eqref{eqn:relac}, re-arranging the terms, and noting that
the terms involving $\|x_{k+1}-x_k\|^2$ get cancelled, we obtain
\[\a_k \langle \tilde g_k, x_k-z\rangle \le D_w(x_k,z)-D_w(x_{k+1},z)+\frac{\a_k^2}{2\mu_w}\|\tilde g_k\|_*^2.\]
By taking the expectation conditioned on $\F_k$ and using Assumption~\ref{asum:stochmir}, we have
\[\a_k \langle g_k, x_k-z\rangle \le D_w(x_k,z)-\EXP{D_w(x_{k+1},z)\mid \F_k}+\frac{\a_k^2\tilde C^2}{2\mu_w},\]
and the desired relation follows.
\end{proof}

With Lemma~\ref{lemma:basic} we are ready to explore the properties of stochastic mirror-descent algorithm
and its weighted-average iterates.
In the following two sections, we will consider two special instances of problem~\eqref{eqn:prob},
namely the case when $f$ is strongly convex but no additional assumptions are made on $X$, and the case when $X$ is bounded (in addition to being convex and closed) but no additional assumptions are made on $f$
aside from convexity, as given in Section~\ref{sec:prelim}.

%%%%%%%%%%%%%%%%%%%%%%%%%%%%%%%%%%%%%%%%
\section{Strongly Convex Objective Function}\label{sec:stronglyconvex}
%%%%%%%%%%%%%%%%%%%%%%%%%%%%%%%%%%%%%%%%
In this section, we restrict our attention to problem~\eqref{eqn:prob} with a strongly convex objective function
$f$. Specifically,  we assume that $f$ is strongly convex with a constant $\mu_f>0$ over the set $X$ with respect to the underlying norm,
\[f(y)\ge f(x)+\langle g(x), y-x\rangle + \frac{\mu_f}{2}\|y-x\|^2\qquad\hbox{for all }x,y\in X,\]
or equivalently
\begin{equation}\label{eqn:strongconv}
\langle g(x), x-y\rangle \ge f(x) - f(y) + \frac{\mu_f}{2}\|x-y\|^2\qquad\hbox{for all }x,y\in X.\end{equation}
For such a function, we consider the stochastic subgradient mirror-descent method with the stepsize
$\frac{\a_k}{\mu_f}$. Specifically, the algorithm assumes the following form:
\begin{equation}\label{eqn:mdstoch2}
x_{k+1}=\argmin_{z\in X}\left\{\frac{\a_k}{\mu_f}\langle \tilde g_k,z-x_k\rangle +D_w(x_k,z)\right\}.
\end{equation}
The stepsize $\a_k$ is assumed to be such that
\begin{equation}\label{eqn:step}
\alpha_{k}\in(0,1]\qquad\hbox{ and }\qquad \frac{1-\alpha_{k+1}}{\alpha^2_{k+1}}\le \frac{1}{\alpha^2_k}
\qquad\hbox{for all }k\ge0,
\end{equation}
where $\alpha_0=1$. The above stepsize choice have been proposed by
Tseng~\cite{Tseng2008} as a generalization of the stepsize selection due to Nesterov~\cite{Nesterov83}
who had developed it in the construction of the optimal algorithm for
minimizing a convex function with Lipschitz gradients
(see also a recent paper by Beck and Teboulle~\cite{Beck2009}).

The following result for the stepsize will be useful in the analysis of the method.

\begin{lemma}\label{lemma:step}
Let the stepsize $\a_k$ satisfy relation~\eqref{eqn:step}. We then have
\[\a_k^2\ge \frac{1}{\sum_{t=0}^k\frac{1}{\a_t}}\qquad\hbox{for all $k\ge0$}.\]
\end{lemma}
\begin{proof}
For $k=0$, the result holds since $\a_0=1$. From relation~\eqref{eqn:step} it follows that
\[\frac{1}{\a_{t+1}^2}-\frac{1}{\a_t^2}\le \frac{1}{\a_{t+1}},\]
which upon summing over $t=0,1,\ldots, k-1,$ yields
\[\frac{1}{\a_{k}^2} - \frac{1}{\a_0^2}\le \sum_{t=1}^k\frac{1}{\a_t}.\]
This and the fact $\a_0=1$ imply the desired relation.
\end{proof}

We next investigate the behavior of the iterates generated by the stochastic mirror-descent algorithm
with a stepsize $\a_k$ satisfying~\eqref{eqn:step}. Our subsequent results rely
on an additional property of the  Bregman distance function $D_w(x,z)$ requiring that
\begin{equation}\label{eqn:dist}
D_w(x,z)\le \frac{1}{2}\|x-z\|^2\qquad\hbox{for all }x,z\in X.\end{equation}
This relation holds for example when the Bregman distance generating function $w$ has
Lipschitz gradients over $X$ with a constant $L=\frac{1}{2}$, i.e.,
\[\langle\nabla w(x)-\nabla w(z),x-z\rangle\le \frac{1}{2}\|x-z\|^2\qquad\hbox{for all }x,z\in X.\]
To see this, note that by the preceding Lipschitz gradient property of $w$, we have for any $x,z\in X$,
\[\frac{1}{2}\|x-z\|^2\ge \langle\nabla w(z)-\nabla w(x),z-x\rangle
\ge w(z)-w(x) -\langle \nabla w(x),z-x\rangle = D(x,z),\]
where the last inequality follows by the convexity of $w$ over $X$, i.e.,
$\langle\nabla w(z), z-x\rangle\ge w(z)-w(x)$ for all $x,z\in X$.
Furthermore, if $\tilde w$ has Lipschitz gradients over $X$ with a scalar $L>0$,
\[\langle\nabla \tilde w(x)-\nabla \tilde w(z),x-z\rangle\le L\|x-z\|^2\qquad\hbox{for all }x,z\in X,\]
then the scaled function $w(x)=\frac{1}{2L}\tilde w(x)$ has Lipschitz gradients with
$L=\frac{1}{2}$. Such a scaled function $w$ can be used for generating
the Bregman distance function satisfying~\eqref{eqn:dist}.
For example, if the underlying norm in $\re^n$ is the Euclidean norm, then choosing
$w(x)=\frac{1}{2}\|x\|_2^2$ would result in $D_w(x,z)=\frac{1}{2}\|x-z\|_2^2$, which satisfies
relation~\eqref{eqn:dist} as equality.

Note that, when $f$ is strongly convex, the minimization problem in~\eqref{eqn:prob}
has a unique minimizer (see Theorem 2.2.6 in~\cite{NesterovBook2004}, or
Proposition 2.1.2 in~\cite{Bertsekas03}), which we denote by $x^*$.
In the following lemma we provide a relation for the iterates $x_k$ and the solution $x^*$.

\begin{lemma}\label{lemma:ratemirdes}
Let $f$ be strongly convex over $X$ with a constant $\mu_f>0$, and
let Assumption~\ref{asum:stochmir} hold.
Also, let the Bregman distance function $D_w$ be such that~\eqref{eqn:dist} holds.
Consider the method~\eqref{eqn:mdstoch2} with
the stepsize sequence $\{\a_k\}$ satisfying the conditions in~\eqref{eqn:step}.
Then, for the iterate sequence $\{x_k\}$ generated by the method and
the solution $x^*$ of problem~\eqref{eqn:prob}, we have
\[\EXP{D_w(x_{k+1},x^*) } +\frac{1}{\mu_f} \,\frac{1}{(\sum_{t=0}^k\frac{1}{\a_t})}\,
\sum_{t=0}^k\frac{1}{\a_{t}}\left(\EXP{f(x_t)} - f(x^*)\right)
\le (k+1)\a_k^2\,\frac{\tilde C^2}{2\mu_f^2\mu_w}.\]
%where $x^*$ is the solution of problem~\eqref{eqn:prob}.
\end{lemma}

\begin{proof}
Under Assumption~\ref{asum:stochmir}, by Lemma~\ref{lemma:basic} (where $\a_k$ is replaced with
$\frac{\a_k}{\mu_f}$)
for method~\eqref{eqn:mdstoch2} we have almost surely for $z=x^*$ and all $k\ge0$,
\[\EXP{D_w(x_{k+1},x^*)\mid \F_k} + \frac{\a_k}{\mu_f} \langle g_k, x_k-x^*\rangle \le D_w(x_k,x^*)
+\frac{\a_k^2\tilde C^2}{2\mu^2_f\mu_w}.\]
By using the strong convexity of $f$ (cf.\ relation~\eqref{eqn:strongconv}), we obtain
\[\frac{\a_{k}}{\mu_f}\langle g_k,x_k-x^*\rangle
\ge \frac{\a_{k}}{\mu_f}\left(f(x_k)-f(x^*)\right)+\frac{\a_{k}}{2}\|x_k-x^*\|^2
\ge \frac{\a_{k}}{\mu_f}\left(f(x_k)-f(x^*)\right)+\a_{k} D_w(x_k,x^*),\]
where the last inequality follows by the assumed property of $D_w$ in Eq.~\eqref{eqn:dist}.
Combining the preceding two relations and taking the total expectation, we further obtain
for all $k\ge0$,
\begin{equation}\label{eqn:maind}
\EXP{D_w(x_{k+1},x^*)} + \frac{\a_{k}}{\mu_f}\left(\EXP{f(x_k)}-f(x^*)\right)
\le (1- \a_{k}) \EXP{D_w(x_k,x^*)}
+\frac{\alpha^2_{k}\tilde C^2}{2\mu_f^2\mu_w}. \end{equation}
Multiplying both sides of this relation by $1/\a^2_{k}$ and using
$\frac{1- \a_{k}}{\a_{k}^2}\le \frac{1}{\a_{k-1}^2}$ (cf.~Eq.~\eqref{eqn:step}), we have for all $k\ge1$,
\begin{eqnarray*}
\frac{1}{\a^2_{k}}\,\EXP{D_w(x_{k+1},x^*)} + \frac{1}{\a_{k}\mu_f}\left(\EXP{f(x_k)}-f(x^*)\right)
&\le& \frac{1- \a_{k}}{\a_{k}^2}\,\EXP{D_w(x_k,x^*)}
+ \frac{\tilde C^2}{2\mu_f^2\mu_w}\cr
&\le& \frac{1}{\a_{k-1}^2}\, \EXP{D_w(x_k,x^*)}
 +\frac{\tilde C^2}{2\mu_f^2\mu_w}.\end{eqnarray*}
Summing these inequalities over $k, k-1,\ldots,1,$ and using $\a_0=1$, we obtain
\begin{equation}\label{eqn:sfour}
\frac{1}{\a^2_{k}}\, \EXP{D_w(x_{k+1},x^*) }
+ \frac{1}{\mu_f}\,\sum_{t=1}^k\frac{1}{\a_{t}}\left(\EXP{f(x_t)} - f(x^*)\right)
\le \EXP{D_w(x_1,x^*)}+k\,\frac{\tilde C^2}{2\mu_f^2\mu_w}.
\end{equation}
We estimate $\EXP{D_w(x_1,x^*)}$ by using relation~\eqref{eqn:maind} with $k=1$
and the fact that $\a_0=1$. Thus, we have for all $k\ge1$,
\begin{equation}\label{eqn:sfive}
\EXP{D_w(x_{1},x^*)}  + \frac{1}{\mu_f}\left( \EXP{f(x_0)} - f(x^*)\right)
\le\frac{\tilde C^2}{2\mu^2_f\mu_w}.\end{equation}
From relations~\eqref{eqn:sfive} and~\eqref{eqn:sfour} we see that for all $k\ge0$,
\[\frac{1}{\a^2_{k}}\,\EXP{D_w(x_{k+1},x^*) }
+\frac{1}{\mu_f}\,\sum_{t=0}^k\frac{1}{\a_{t}}\left( \EXP{f(x_t)} - f(x^*)\right)
\le (k+1)\frac{\tilde C^2}{2\mu^2_f\mu_w},\]
from which the desired relation follows by multiplying with $\a_{k}^2$ and using the
relation $\a_k^2\ge \frac{1}{\sum_{t=0}^k\frac{1}{\a_t}}$, which holds for all $k\ge0$
by virtue of Lemma~\ref{lemma:step}.
\end{proof}

Lemma~\ref{lemma:ratemirdes} indicates that we can state some special result for a weighted-average
points of the method, defined as follows:
\begin{equation}\label{eqn:aver}
\hat x_k=\frac{1}{(\sum_{t=0}^k\frac{1}{\a_t})}\,\sum_{t=0}^k\frac{1}{\a_{t}} x_t
\qquad\hbox{for all } k\ge0.\end{equation}
Note that $\hat x_k\in X$ for all $k$, as each $\hat x_k$ is a convex combination of points in the set $X$.
From Lemma~\ref{lemma:ratemirdes} we see that
\[\EXP{f(\hat x_k)}-f(x^*)\le (k+1)\a_k^2\,\frac{\tilde C^2}{2\mu_f\mu_w}\qquad\hbox{for all }k\ge0.\]
To make this estimate more meaningful, we turn our attention to the stepsize choices that satisfy the conditions in~\eqref{eqn:step}. From~\eqref{eqn:step}  we obtain
\begin{equation*}%\label{eqn:step_paul1}
0<\alpha_{k+1}\le\frac{\sqrt{\alpha_k^4+4\alpha^2_k}-\alpha_k^2}{2}\qquad\hbox{for all }k\ge0,
\end{equation*}
starting with $\a_0=1$.
We will consider two specific choices, namely
\begin{equation}\label{eqn:step_paul}
	\alpha_k=\frac{2}{k+1}\qquad\hbox{for all }k\ge0,
\end{equation}
and
\begin{equation}\label{eqn:step_paul2}
\alpha_{k+1}=\frac{\sqrt{\alpha_k^4+4\alpha^2_k}-\alpha_k^2}{2}\qquad\hbox{for all }k\ge0.
\end{equation}
The stepsize in~\eqref{eqn:step_paul} has been proposed by Tseng~\cite{Tseng2008}.
Setting $\alpha_{k+1}=\frac{1}{t_{k+1}}$ in~\eqref{eqn:step_paul2}
will yield the Nesterov sequence $t_{k+1}$ used in the construction of the fastest first-order
algorithm for convex functions with Lipschitz gradients, i.e.,
\[t_{k+1}=\frac{\sqrt{1+4t^2_k}+1}{2}\qquad\hbox{for all }k\ge0,\]
with $t_0=1$ (see Nesterov~\cite{Nesterov83}, also~Beck and Teboulle~\cite{Beck2009}).
By induction, it can be seen that $t_{k}\ge\frac{k+2}{2}$ for all $k$, implying that
the stepsize in~\eqref{eqn:step_paul2} satisfies
\begin{equation}\label{eqn:stepcond}
0<\alpha_k\le\frac{2}{k+1}\quad\hbox{for all }k\ge0\qquad\hbox{with }\a_0=1.
\end{equation}
Thus, both stepsize choices in~\eqref{eqn:step_paul} and~\eqref{eqn:step_paul2}
satisfy relation~\eqref{eqn:stepcond}.
Note that these stepsize choices do not have any tunable parameters.

Now, we provide the convergence result for algorithm~\eqref{eqn:mdstoch2}
with the aforementioned stepsize choices.

\begin{theorem}\label{prop:ratemirdes}
Let $f$ be strongly convex
over $X$ with a constant $\mu_f>0$, and let Assumption~\ref{asum:stochmir} hold.
Also, let the Bregman distance function $D_w$ be such that~\eqref{eqn:dist} holds.
Consider the method~\eqref{eqn:mdstoch2} with
the stepsize sequence $\{\a_k\}$ chosen according to either~\eqref{eqn:step_paul}
or~\eqref{eqn:step_paul2}.
Then, for the iterate sequence $\{x_k\}$ generated by the method,
we have
\[\EXP{\|x_{k+1}-x^* \|^2}
\le \frac{4}{k+1}\,\frac{\tilde C^2}{\mu_f^2\mu_w^2}\qquad\hbox{for all }k\ge0,\]
while for the weighted-average sequence $\{\hat x_k\}$, with $\hat x_k$ defined
in~\eqref{eqn:aver}, we have
\[\EXP{f(\hat x_k)} - f(x^*)
\le \frac{2}{k+1}\,\frac{\tilde C^2}{\mu_f\mu_w}\qquad\hbox{for all }k\ge0,\]
\[\EXP{\|\hat x_k - x^*\|^2}
\le \frac{4}{k+1}\,\frac{\tilde C^2}{\mu_f^2\mu_w}\qquad\hbox{for all }k\ge0,\]
where $x^*$ is the solution of problem~\eqref{eqn:prob}.
\end{theorem}

\begin{proof}
By Lemma~\ref{lemma:ratemirdes}
we have for all $k\ge0$,
\[\EXP{D_w(x_{k+1},x^*) }
+\frac{1}{\mu_f}\,\frac{1}{(\sum_{t=0}^k\frac{1}{\a_t}) }\,
\sum_{t=0}^k\frac{1}{\a_{t}}\left( \EXP{f(x_t)} - f(x^*)\right)
 \le (k+1)\a_k^2\,\frac{\tilde C^2}{2\mu_f^2\mu_w}.\]
The stepsize sequence $\{\a_k\}$ chosen according to either~\eqref{eqn:step_paul}
or~\eqref{eqn:step_paul2} satisfies
$\a_k\le \frac{2}{k+1}$ for all $k$ (cf.~Eq.~\eqref{eqn:stepcond}).
Thus, for all $k\ge0$,
\begin{equation}\label{eqn:finale}
\EXP{D_w(x_{k+1},x^*) }
+\frac{1}{\mu_f} \,\frac{1}{(\sum_{t=0}^k\frac{1}{\a_t})}\,
\sum_{t=0}^k\frac{1}{\a_{t}}\left(\EXP{f(x_t)} - f(x^*)\right)
\le \frac{2}{k+1}\,\frac{\tilde C^2}{\mu_f^2\mu_w}.\end{equation}
By using the convexity of the function $f$ and
the definition of the weighted-average $\hat x_k$ in Eq.~\eqref{eqn:aver}, we conclude that
\[\EXP{f(\hat x_k)} - f(x^*)
\le \frac{2}{k+1}\,\frac{\tilde C^2}{\mu_f\mu_w}\qquad\hbox{for all }k\ge0,\]
while by the strong convexity of $f$, we have $f(\hat x_k)-f(x^*)\ge\frac{\mu_f}{2}\|\hat x_k-x^*\|^2$
which yields
\[\EXP{\|\hat x_k - x^*\|^2}
\le \frac{4}{k+1}\,\frac{\tilde C^2}{\mu_f^2\mu_w}\qquad\hbox{for all }k\ge0.\]
Furthermore, from relation~\eqref{eqn:finale} we obtain
\[\EXP{D_w(x_{k+1},x^*) }
\le \frac{2}{k+1}\,\frac{\tilde C^2}{\mu_f^2\mu_w}\qquad\hbox{for all }k\ge0,\]
which by the strong convexity of $w$ (Eq.~\eqref{eqn:bdstr}) yields
\[\EXP{\|x_{k+1}-x^* \|^2}
\le \frac{4}{k+1}\,\frac{\tilde C^2}{\mu_f^2\mu_w^2}\qquad\hbox{for all }k\ge0.\]
\end{proof}

Theorem~\ref{prop:ratemirdes} gives the rate of convergence of the order $\frac{1}{k}$
for the expected distance of the averages
$\hat x_k$ to the solution $x^*$. This rate is known to be the optimal rate achievable
by a stochastic subgradient methods for strongly convex functions, as shown
in~\cite{GhadimiLan2012} and~\cite{Rakhlin2012}.
Specifically, in~\cite{GhadimiLan2012}, the optimal rate is attained by a more involved method
(utilizing averages in a different way), while in~\cite{Rakhlin2012} the rate is attained by suitably sliding
the window of indices over which the iterates are averaged.
We note, however, that Theorem~\ref{prop:ratemirdes} is not as general as the results obtained
in~\cite{GhadimiLan2012}, which can simultaneously handle the case when the function $f$ may have Lipschitz gradients.

The method~\eqref{eqn:mdstoch2} is simple for implementation, as it requires
storing  $x_k$, the weighted average $\hat x_{k-1}$ (initialized with $\hat x_0=x_0$) and
the stepsize-related sum $S_k=\sum_{t=0}^k\frac{1}{\a_t}$ at each iteration.
At iteration $k+1$, the weighted average $\hat x_{k-1}$ gets updated. This can be done recursively
by computing $S_{k+1}=S_k+\frac{1}{\a_k}$, starting with $S_0=0$, and  by setting
\[\hat x_{k}=\frac{S_{k}}{S_{k+1}}\hat x_{k-1} +\left(1-\frac{S_k}{S_{k+1}}\right) x_k.\]

Theorem~\ref{prop:ratemirdes} provides the convergence rate for the expected distances
$\EXP{\|x_k - x^*\|^2}$ for the iterates of the algorithm, but not for their expected function values,
while for the averaged iterates $\hat x_k$, it provides both estimates for the expected distances
and the expected function values. Furthermore, observe that for the expected distances, the estimate
for $\EXP{\|x_k - x^*\|^2}$ scales proportionally to $\mu_w^{-2}$ while
the estimate for $\EXP{\|\hat x_k - x^*\|^2}$ scales proportionally to $\mu_w^{-1}$.
When $\mu_w=1$, such as in the case of the Euclidean
norm and $D_w(x,z)=\frac{1}{2}\|x-z\|_2^2$, then both of these expected distance estimates
are the same, whereas the estimate for  $\EXP{\|\hat x_k - x^*\|^2}$ is better when $\mu_w\in (0,1)$.

In the case when the subgradients $g_k$ are evaluated without any errors, corresponding to
$\tilde g_k=g_k$ in algorithm~\eqref{eqn:mdstoch}, from Theorem~\ref{prop:ratemirdes}
we obtain the following immediate result.

\begin{corollary}
Assume that the subgradients $g(x)$ are uniformly bounded over the set $X$, i.e.,
there is a scalar $C$ such that $\|g(x)\|\le C$ for all $x\in X$.
Let $f$ be strongly convex over $X$ with a constant $\mu>0$.
Also, let the Bregman distance function $D_w$ satisfy~\eqref{eqn:dist}.
Consider the method~\eqref{eqn:mdstoch2} with $\tilde g(x)=g(x)$ and
the stepsize sequence $\{\a_k\}$ chosen according to either~\eqref{eqn:step_paul}
or~\eqref{eqn:step_paul2}.
Then, for the iterate sequence $\{x_k\}$ and the solution $x^*$ of problem~\eqref{eqn:prob}we have
\[\|x_{k+1}-x^* \|^2
\le \frac{4}{k+1}\,\frac{C^2}{\mu_f^2\mu_w^2}\qquad\hbox{for all }k\ge0,\]
while for the weighted-average sequence $\{\hat x_k\}$ we have
\[f(\hat x_k) - f(x^*)
\le \frac{2}{k+1}\,\frac{C^2}{\mu_f\mu_w},\qquad
\|\hat x_k - x^*\|^2
\le \frac{4}{k+1}\,\frac{C^2}{\mu_f^2\mu_w}\qquad\hbox{for all }k\ge0.\]
\end{corollary}

Note that Theorem~\ref{prop:ratemirdes} does not say anything about
the convergence of $\hat x_k$ to the solution $x^*$. However, this can be established
using an analysis similar to that of stochastic approximation
methods~\cite{Ermoliev76,Ermoliev83,Ermoliev88,Borkar08,Polyak87}.
Our convergence analysis is based on a result of Lemma 10, pages 49--50, in~\cite{Polyak87},
which is stated below.

 \begin{lemma}[Lemma 10, \cite{Polyak87}]\label{lemma:polyak}
 Let $\{v_k\}$ be a sequence of non-negative random variables with $\EXP{v_0}<\infty$ and
 such that the following holds almost surely
 \[\EXP{v_{k+1}\mid v_0,\ldots,v_k}\le (1-\alpha_k) v_k+\beta_k\qquad\hbox{for all }k\ge0,\]
 where $\{\a_k\}$ and $\{\beta_k\}$ are deterministic non-negative scalar sequences satisfying
 $0\le \alpha_k\le 1$ for all $k$ and
 \[\sum_{k=0}^\infty\alpha_k=\infty, \qquad \sum_{k=0}^\infty\beta_k<\infty,\qquad
 \lim_{k\to\infty}\frac{\beta_k}{\alpha_k}=0.\]
 Then $\lim_{k\to\infty} v_k=0$ almost surely.
 \end{lemma}
%\todo{LEMMA also gives an upper bound on the probability! Shall it be used?}
Using Lemma~\ref{lemma:polyak},
we establish the almost sure convergence of the iterates and their averages in the following.

\begin{theorem}\label{thm:asconv}
Under the assumptions of Theorem~\ref{prop:ratemirdes},
the iterates $x_k$ and their weighted averages $\hat x_k$, as defined in~\eqref{eqn:aver},
converge to the optimal point $x^*$ almost surely.
\end{theorem}
\begin{proof}
We start with relation~\eqref{eqn:maind}, as shown in the proof of Lemma~\ref{lemma:ratemirdes}:
\begin{equation}\label{eqn:mainda}
\EXP{D_w(x_{k+1},x^*)}
\le (1- \a_{k}) \EXP{D_w(x_k,x^*)} - \frac{\a_{k}}{\mu_f}\left(f(x_k)-f(x^*)\right)
+\frac{\alpha^2_{k}}{2\mu_f^2\mu_w}\, \tilde C^2. \end{equation}
Both stepsize choices~\eqref{eqn:step_paul} and~\eqref{eqn:step_paul2} are such that
$\sum_{k=0}^\infty \a_k=\infty$ and $\sum_{k=0}^\infty \a_k^2 <\infty$. Thus, we can apply
Lemma~\ref{lemma:polyak} (where $\beta_k=\alpha_k^2$) to conclude that $x_k\to x^*$ almost surely.
Furthermore, since $\hat x_k$ is a convex combination of $x_0,\ldots,x_k$ for each $k$, the vectors
$\hat x_k$ must also converge to $x^*$ almost surely.
\end{proof}

The convergence of the stochastic gradient method has been known for strongly convex functions
with Lipschitz gradients (see Theorem 2, pages 52--53, in~\cite{Polyak87}). Theorem~\ref{thm:asconv}
extends this result to a more general class of stochastic subgradient methods and strongly convex functions
that are not necessarily differentiable.
%\todo{Probabilistic bound can be obtained similar to Rakhlin! Section B7.}

%-------------------------------------------------------------------
\section{Compact Constraint Set}\label{sec:compact}
%-------------------------------------------------------------------
We now consider the stochastic subgradient mirror-descent method of~\eqref{eqn:mdstoch} for the case when the function $f$ is convex while the set $X$ is {\it bounded} in addition to being closed and convex.
Under our blanket assumptions of Section~\ref{sec:prelim}, the compactness of $X$ need not imply
the uniform boundedness of the subgradients of $f$ for $x\in X$, since we did not
impose any suitable condition on the domain of $f$ to ensure this property.
Thus, we will impose this condition explicitly.

\begin{assumption}\label{asum:compact}
Let $X$ be compact and assume there exists a scalar $C>0$ such that
\[ \|g(x)\|_*\le C\qquad\hbox{for all }x\in X.\]
\end{assumption}
In addition, we assume the following for the stochastic subgradient errors.

\begin{assumption}\label{asum:stochmodif}
Let the stochastic subgradient errors be such that
\[\EXP{\tilde g(x)\mid x}=g(x)\qquad\hbox{for all }x\in X,\]
\[\EXP{\|\tilde g(x)- g(x)\|_*^2\mid x}\le \nu^2
\qquad\hbox{for all $x\in X$ and for some scalar $\nu>0$}.\]
\end{assumption}

Under these two assumptions, we can see that for all $x\in X$,
\[\EXP{\|\tilde g(x)\|_*^2\mid x}\le \EXP{(\|\tilde g(x)- g(x)\|_*+\|g(x)\|_*)^2\mid x}
%\le 2\EXP{\|\tilde g(x)- g(x)\|_*^2+\|g(x)\|_*^2\mid x}
\le2(\nu^2+C^2),\]
implying that Assumption~\ref{asum:stochmir} is satisfied with $\tilde C^2=2(\nu^2+C^2).$

In parallel with the mirror-descent iterates, we consider the weighted averages $\hat x_k$
of the iterates $x_0,x_1,\ldots,x_k$, as given in~\eqref{eqn:aver}.
%\begin{equation}\label{eqn:aver2}
%\[\hat x_k=\frac{1}{(\sum_{t=0}^k\frac{1}{\a_t})}\,\sum_{t=0}^k\frac{1}{\a_{t}} x_t
%\qquad\hbox{for all } k\ge0. \] %\end{equation}
We have the following basic relation for the weighted iterate-averages.
\begin{lemma}\label{lemma:basicstochmod}
Let Assumptions~\ref{asum:compact} and~\ref{asum:stochmodif} hold.
Also, assume that the stepsize $\a_k$ is non-increasing, i.e., $\a_k\le \a_{k-1}$ for all $k\ge1$. Then,
for the weighted averages $\hat x_k$ of the iterates generated by
algorithm~\eqref{eqn:mdstoch} there holds for all $z\in X$ and $k\ge0$,
\[%\frac{1}{(\sum_{t=0}^k\frac{1}{\a_{t}})}\,\sum_{t=0}^k\frac{1}{\a_{t}}\left(\EXP{f(x_t)}-f(z)\right)
 \EXP{f(\hat x_k)}-f(z)\le \frac{1}{\sum_{t=0}^k\frac{1}{\a_{t}}}
 \left(\frac{d_w^2}{\a_{k}^2} +(k+1)\,\frac{C^2 +\nu^2}{\mu_w}\right),\]
 where $d^2_w=\max_{x,y\in X} D_w(x,y)$.
\end{lemma}

\begin{proof}
Assumptions~\ref{asum:compact} and~\ref{asum:stochmodif} imply Assumption~\ref{asum:stochmir}.
Then, by Lemma~\ref{lemma:basic} with $\tilde C^2=2(C^2+\nu^2)$,
for method~\eqref{eqn:mdstoch} we have almost surely for all $z\in X$ and $k\ge0$,
\begin{equation}\label{eqn:beg}
\EXP{D_w(x_{k+1},z)\mid \F_k}+\a_k \langle g_k, x_k-z\rangle \le D_w(x_k,z)
+\frac{\a_k^2(C^2+\nu^2)}{\mu_w}.
\end{equation}
Dividing the whole inequality with $\a_k^2$ and re-arranging the terms, we obtain
\begin{eqnarray*}
\frac{1}{\a^2_{k}}\,\EXP{D_w(x_{k+1},z)\mid \F_k}
\le \frac{1}{\a_k^2}\, D_w(x_k,z) - \frac{1}{\a_{k}}\langle g_k,x_k-z\rangle
+\frac{C^2+\nu^2}{\mu_w}.\end{eqnarray*}
We now re-write the term $\frac{1}{\a_k^2}\, D_w(x_k,z)$ for $k\ge 1$, as follows
\begin{eqnarray*}
\frac{1}{\a_k^2}D_w(x_k,z)
&=&\frac{1}{\a_{k-1}^2}D_w(x_k,z) +\left(\frac{1}{\a_k^2}-\frac{1}{\a_{k-1}^2}\right) D_w(x_k,z)\cr
&\le&\frac{1}{\a_{k-1}^2}D_w(x_k,z) +\left(\frac{1}{\a_k^2}-\frac{1}{\a_{k-1}^2}\right)d_w^2,\end{eqnarray*}
where we use $\a_k\le \a_{k-1}$ and
$d^2_w=\max_{x,y\in X} D_w(x,y)$, which is finite since $X$ is compact and $D_w$ is continuous over $X$.
By combining the preceding two relations, after re-arranging the terms, we obtain almost surely
for all $z\in X$ and all $k\ge1$,
\begin{align*}
\frac{1}{\a_k^2}\,\EXP{D_w(x_{k+1},z)\mid \F_k} +\frac{1}{\a_{k}}\langle g_k,x_k-z\rangle
& \le \frac{1}{\a_{k-1}^2}\, D_w(x_k,z) +\left(\frac{1}{\a_k^2}-\frac{1}{\a_{k-1}^2}\right)d_w^2
+\frac{C^2+\nu^2}{\mu_w}.\end{align*}
By using the subgradient property (cf.~\eqref{eqn:strongconv} with $\mu=0$) and by taking the total expectation, we  obtain for all $z\in X$  and for all $k\ge1$,
\begin{eqnarray}\label{eqn:sthreemod}
\frac{1}{\a_k^2}\,\EXP{D_w(x_{k+1},z)}
+ \frac{1}{\a_{k}}\left(\EXP{f(x_k)}-f(z)\right)
&\le& \frac{1}{\a_{k-1}^2}\EXP{D_w(x_k,z)} \cr
&&+ \left(\frac{1}{\a_k^2}-\frac{1}{\a_{k-1}^2}\right)d_w^2
+\frac{C^2+\nu^2}{\mu_w}.\qquad
\end{eqnarray}

Summing the inequalities in~\eqref{eqn:sthreemod} over $k, k-1,\ldots,1,$ and using $\a_0=1$, we obtain
for all $z\in X$ and $k\ge1$,
\begin{equation}\label{eqn:medio}
\frac{1}{\a^2_{k}}\, \EXP{D_w(x_{k+1},z)} + \sum_{t=1}^k\frac{1}{\a_{t}}\left(\EXP{f(x_t)} - f(z)\right)
\le \EXP{D_w(x_1,z)} +\left(\frac{1}{\a_k^2}-1\right) d_w^2+\frac{k(C^2+\nu^2)}{\mu_w}.
\end{equation}
We next estimate $\EXP{D_w(x_1,z)}$ by using relation~\eqref{eqn:beg} with $k=0$ and the fact $\a_0=1$.
Upon using the convexity of $f$ and by taking the total expectation, from~\eqref{eqn:beg} (for $k=0$)
we obtain
\[\EXP{D_w(x_{1},z)} + \frac{1}{\a_0}\EXP{f(x_0)} - f(z)\le \EXP{D_w(x_0,z)}+\frac{C^2+\nu^2}{\mu_w}.\]
Since $D_w(x_0,z)\le d_w^2$, it follows that $\EXP{D_w(x_0,z)}\le d_w^2$,
thus implying that for all $z\in X$,
\[\EXP{D_w(x_{1},z)} + \frac{1}{\a_0}\EXP{f(x_0)} - f(z)\le d_w^2+\frac{C^2+\nu^2}{\mu_w}.\]
Using the preceding estimate in Eq.~\eqref{eqn:medio}, we see that for all $z\in X$ and all $k\ge0$,
\[\frac{1}{\a^2_{k}}\,\EXP{D_w(x_{k+1},z)}
+ \sum_{t=0}^k\frac{1}{\a_{t}}\left( \EXP{f(x_t)} - f(z)\right)
\le\frac{1}{\a_k^2} d_w^2+ (k+1)\frac{C^2+\nu^2}{\mu_w}.\]
The desired relation follows upon dividing the preceding inequality with
$\sum_{t=0}^k\frac{1}{\a_{t}}$, using
the convexity of $f$ and dropping the first term on the left-hand side.
\end{proof}

From Lemma~\ref{lemma:basicstochmod} we see that the upper bound on %for the weighted-averages
$\EXP{f(\hat x_k)}-f(z)$
%\le \frac{\frac{1}{\a_k^2} D^2+ (k+1)\frac{C^2+\nu^2}{\mu_w}}
%{\sum_{t=0}^k\frac{1}{\a_t}}\qquad\hbox{for all $k\ge0$}.\]
%The term on the right-hand side
will converge to zero provided that
\[\lim_{k\to\infty}\frac{\frac{1}{\a_k^2} +(k+1)}  {\sum_{t=0}^k\frac{1}{\a_t}}=0.\]
In what follows, we will investigate the stepsize selection that makes the preceding
convergence to zero as fast as possible.
It is known that the best achievable rate for the standard subgradient method, with or without stochastic errors, is of the order $\frac{1}{\sqrt{k}}$ when $f$ is just a convex function (see the discussion following Theorem 3.2.2 in~\cite{NesterovBook2004} for the subgradient method,
and~\cite{Nemirovski2009} for the stochastic subgradient method).
The aforementioned work shows the estimate and the optimal stepsize selection,
under the premise that a number of iterations is fixed a priori. This rate is shown to be achievable
per iteration by the primal-dual averaging subgradient method recently proposed by
Nesterov~\cite{Nesterov2009}.

We consider a different stepsize choice that will achieve the same rate of $\frac{1}{\sqrt{k}}$.
In particular, let
\begin{equation}\label{eqn:rootstep}
\a_k=\frac{a}{\sqrt{k+1 }}\qquad\hbox{for all }k\ge0,
\end{equation}
where $a>0$ is a parameter, which we will select later.
In this case, the sum $\sum_{t=0}^k\frac{1}{\a_{t}}$ can be bounded from below as given in the following lemma.

\begin{lemma}\label{lemma:sumbound}
For the stepsize $\a_k$ in~\eqref{eqn:rootstep}, we have
$\sum_{t=0}^k\frac{1}{\a_{t}}\ge \frac{2}{3a}(k+1)^{3/2}$ for all $k\ge0$.
\end{lemma}
\begin{proof}
Viewing the sum as an upper-estimate of the integral of the function $t\mapsto\sqrt{t+1}$, we see that
for any $k\ge0$,
\[\sum_{t=0}^k\frac{1}{\a_{t}}=\frac{1}{a}\sum_{t=0}^k\sqrt{t+1}
\ge \frac{1}{a}\int_{t=0}^{k+1}\sqrt{t+1} \,dt=\frac{2}{3a}\left( (k+2)^{3/2}-1\right).\]
We claim that $(k+2)^{3/2}-1\ge (k+1)^{3/2}$ for all $k\ge0$. To show this, we let $s=k+1$ and consider the scalar function $\phi(s)=(s+1)^{3/2}-1-s^{3/2}$ for $s\ge1$. Thus, to prove the claim, it suffices to show that
$\phi(s)\ge0$ for all $s\ge1$.  We note that $\phi(0)=0$.
Furthermore, for the derivative of $\phi$ we have
\[\phi'(s)=\frac{3}{2}\left(\sqrt{s+1}-\sqrt{s}\right)\ge0\qquad\hbox{for all }s\ge0.\]
Thus, $\phi(s)$ is increasing over the interval $[0+\infty)$, and since $\phi(1)=0$, it follows that
$\phi(s)\ge0$ for all $s\ge1$.
\end{proof}

When $X$ is compact, due to our basic assumption that $f$ is continuous over $X$, by Weierestrass theorem
the solution set $X^*$ for problem~\eqref{eqn:prob} is not empty.
For the algorithm~\eqref{eqn:mdstoch} with stepsize~\eqref{eqn:rootstep}, we have the following result
as immediate consequence of Lemmas~\ref{lemma:basicstochmod} and~\ref{lemma:sumbound}.

\begin{theorem}\label{prop:bounded}
Let Assumptions~\ref{asum:compact} and~\ref{asum:stochmodif} hold.
Consider algorithm~\eqref{eqn:mdstoch} with stepsize $\a_k=\frac{a}{\sqrt{k+1 }}$.
Then, for the weighted averages $\hat x_k$ of the iterates produced by the algorithm
we have for all $x^*\in X^*$ and $k\ge0$,
\[\EXP{f(\hat x_k)}-f(x^*)\le \frac{3}{2\sqrt{k+1}}\left(\frac{d_w^2}{a}+\frac{a(C^2+\nu^2)}{\mu_w}\right),
%\frac{3}{4}\frac{\frac{d_w^2}{a}+a(C^2+\nu^2)}{\sqrt{k+1}},
\]
where $d^2_w=\max_{x,y\in X} D_w(x,y)$.
\end{theorem}
\begin{proof}
Since $\a_k$ is non-increasing, by Lemma~\ref{lemma:basicstochmod} (with $z=x^*$),
we find that for any $x^*\in X^*$  and all $k\ge0$,
\[\EXP{f(\hat x_k)}-f(x^*) \le
\frac{\frac{1}{\a_{k}^2}d_w^2 +(k+1)\,\frac{C^2 +\nu^2}{\mu_w}}{\sum_{t=0}^k\frac{1}{\a_{t}}}.\]
By letting $\a_k=a/\sqrt{k+1}$ and using Lemma~\ref{lemma:sumbound}, we obtain
\begin{eqnarray*}
\EXP{f(\hat x_k)}-f(x^*)
&\le &\frac{ \frac{(k+1)}{a^2}d_w^2 +(k+1)\,\frac{C^2 +\nu^2}{2\mu_w}}{\sum_{t=0}^k\frac{1}{\a_{t}}}\cr
&\le & \frac{(k+1)\left(\frac{d_w^2}{a^2}+\frac{C^2+\nu^2}{\mu_w}\right)}{\frac{2}{3a}(k+1)^{3/2}} \cr
&= &\frac{3}{2\sqrt{k+1}}\left(\frac{d_w^2}{a}+\frac{a(C^2+\nu^2)}{\mu_w}\right).\end{eqnarray*}
\end{proof}

Theorem~\ref{prop:bounded} shows that the expected function value evaluated at the weighted average
$\hat x_k$ achieves the rate of $\frac{1}{\sqrt{k} }$, which is known to be the best for a class of convex
functions (not necessarily differentiable). This rate is also achieved by a stochastic primal-dual
averaging method of Nesterov, as shown in~\cite{Xiao2010}, as well as with a stochastic subgradient
mirror-descent with a different form of averaging~\cite{GhadimiLan2012}.
Let us note that Theorem~\ref{prop:bounded} assumes that the set $X$ is compact. This requirement is needed when a time-varying stepsize is used; see also the window-based averaging with
time-varying stepsize in~\cite{GhadimiLan2012}. Intuitively, the boundedness of $X$ is needed to ensure
the boundedness of the iterate sequence $\{x_k\}$, as the stepsize $\a_k=\frac{1}{\sqrt{k} }$ is ``too large"
to ensure that $\{x_k\}$ is bounded (almost surely)
even when the subgradients and their noise variance are bounded.
However, when the number $N$ of iterations is fixed a priori
and the constant optimal stepsize is employed, the compactness of $X$ is not needed as the
finite sequence $\{x_k, \, 0\le k\le N\}$ is evidently always bounded.

Next, we focus on the best selection
of the parameter $a>0$. When estimates for the subgradient-norm bound $C$, the diameter
$d_w$ of the set $X$, and the bound $\nu$ of the expected error norm are available,
we can select the parameter $a$ optimally by minimizing the term
$\frac{d_w^2}{a}+a\frac{C^2+\nu^2}{\mu_w}$ as a function of $a$, over $a>0$.
By doing so, we find that the minimum of the function
$a\mapsto\frac{d_w^2}{a}+a\frac{C^2+\nu^2}{\mu_w}$ over $a>0$ is attained at $a^*=\frac{d_w}{\sqrt{\frac{C^2+\nu^2}{\mu_w}} },$ with the minimum value of $d_w\sqrt{\frac{C^2+\nu^2}{\mu_w} }$.
Thus, when $a$ in~\eqref{eqn:rootstep} is selected optimally, the result of Theorem~\ref{prop:bounded} reduces to
\begin{equation}\label{eqn:opt}
\EXP{f(\hat x_k)}-f^*\le \frac{3}{2\sqrt{\mu_w}}\frac{d_w\sqrt{C^2+\nu^2}}{\sqrt{k+1}}
\qquad\hbox{for all } k\ge0.\end{equation}
In particular, if the function $w$ is the Euclidean norm, i.e., $w(x)=\frac{1}{2}\|x\|^2_2$, then $\mu_w=1$,
$D_w(x,z)=\frac{1}{2}\|x-z\|_2^2$, and
the preceding bound would reduce to (by using $\tilde C^2=C^2+\nu^2$ in the proof of
Lemma~\ref{lemma:basicstochmod}):
\[\EXP{f(\hat x_k)}-f^*\le \frac{3}{2\sqrt{2}}\frac{(\max_{x,y\in X}\|x-y\|_2) \sqrt{C^2+\nu^2}}{\sqrt{k+1}}
\qquad\hbox{for all } k\ge0.\]

We now consider the method in~\eqref{eqn:mdstoch} in the absence of errors ($\tg_k=g_k$),
which corresponds to the standard subgradient mirror-descent method~\cite{NemYudin1978,Beck2003}.
Then, as an immediate consequence of
Theorem~\ref{prop:bounded}, we have the following result.

\begin{corollary}\label{cor:bounded}
Let Assumption~\ref{asum:compact} hold, and let the subgradients $g_k=g(x_k)$ be used in
method~\eqref{eqn:mdstoch} instead of stochastic subgradients $\tg_k$.
Also, assume that the stepsize
$\a_k$ is given by~\eqref{eqn:rootstep}. Then,
for the weighted averages $\hat x_k$ of the iterates produced by algorithm~\eqref{eqn:mdstoch}
there holds %for all $x^*\in X^*$ and
for all $k\ge0$,
\[f(\hat x_k)-f^*\le \frac{3}{2\sqrt{k+1}}\left(\frac{d_w^2}{a}+\frac{aC^2}{2\mu_w} \right).\]
%where $d_w=\max_{x,y\in X} D_w(x,y)$.
Furthermore, every accumulation point of
$\{\hat x_k\}$ is a solution to problem~\eqref{eqn:prob}. Moreover, when the parameter $a$ is selected
so as to minimize the right-hand side of the preceding relation, then the optimal choice is
$a^*=\frac{d_w\sqrt{2\mu_w}}{C}$ and the corresponding error estimate is given by
\[f(\hat x_k)-f^*\le \frac{3}{2\sqrt{2\mu_w}}\frac{d_wC}{\sqrt{k+1}}\qquad\hbox{for all } k\ge0.\]
\end{corollary}
\begin{proof}
The given estimates follow from Theorem~\ref{prop:bounded} and relation~\eqref{eqn:opt}, respectively,
by letting $\nu=0$.
The statement about the accumulation points of $\{\hat x_k\}$ follows from the boundedness of
$\{\hat x_k\}$ (due to $X$ being compact)
and our basic underlying assumption on continuity of $f$ at all points $x\in X$.
\end{proof}

An interesting insight from Corollary~\ref{cor:bounded} is that, while we have no guarantees that the accumulation points of the iterate sequence $\{x_k\}$ are related to the solutions of problem~\eqref{eqn:prob}
or not,  all the accumulation points of the weighed averages $\hat x_k$ are solutions of the problem.
Specifically, our stepsize $\a_k=\frac{a}{\sqrt{k+1 }}$ does not satisfy the standard conditions
that ensure the convergence of $\{x_k\}$ to solution set, namely, $\sum_{k=0}^\infty \a_k=\infty$ and
$\sum_{k=0}^\infty \a_k^2<\infty$. Thus, we do not have a basis to assert that the
accumulation points of $\{x_k\}$ lie in the solution set $X^*$.
However, Corollary~\ref{cor:bounded} shows that the accumulation points of the averaged sequence
$\{\hat x_k\}$ belong to the solution set.

Corollary~\ref{cor:bounded} shows that the optimal error rate can be achieved with a subgradient
mirror-descent method augmented with an outside weighted averaging of the iterates.
It is an alternative optimal method in addition to the primal-dual averaging subgradient method of Nesterov~\cite{Nesterov2009}.

Now, we extend the convergence result for $\hat x_k$ stated in Theorem~\ref{prop:bounded}.
As noted earlier, the stepsize in~\eqref{eqn:rootstep} does not satisfy the standard
condition $\sum_{k}\a_k^2<\infty$ that is typically used to establish the almost sure
convergence of $\{x_k\}$ to the solution set. Thus, Theorem~\ref{prop:bounded}
does not provide us with such information.
However, Theorem~\ref{prop:bounded} can be used as a starting point, which leads to the following result.

\begin{theorem}\label{prop:boundedconv}
Under the assumptions of Theorem~\ref{prop:bounded}, for the weighted-average sequence $\{\hat x_k\}$
of the iterates generated by method~\eqref{eqn:mdstoch}, we have almost surely
\[\liminf_{k\to\infty} f(\hat x_k)=f^*,\qquad\liminf_{k\to\infty} {\rm dist} (\hat x_k,X^*) =0,\]
%and one of the (random) accumulation points of $\{\hat x_k\}$ is a solution of
%problem~\eqref{eqn:prob} almost surely.
where ${\rm dist}(x,Y)$ denotes the distance from a point $x$ to a set $Y\subset \re^n$.
Furthermore, almost surely it holds
\[\lim_{k\to\infty} \left(\min_{0\le t\le k} f(x_t)-f^*\right)=0,\qquad
\lim_{k\to\infty} \left(\min_{0\le t\le k}{\rm dist} (x_t,X^*) \right)=0.\]
\end{theorem}
\begin{proof}
By Theorem~\ref{prop:bounded} we have
$\lim_{k\to\infty}\left(\EXP{f(\hat x_k)}-f^*\right)=0,$
which by Fatou's lemma and $f(\hat x_k)-f^*\ge 0$ for all $k$ yields
%where
%$K=\frac{3}{2}\left(\frac{d_w^2}{a}+\frac{a(C^2+\nu^2)}{\mu_w}\right).$
%By summing over $k=0,\ldots,\ell$ for an arbitrary integer $\ell>0$, and then, by letting $\ell\to\infty$ and using the monotone convergence theorem (see for example
%Theorem 16.2 in~\cite{Billingsley95}), we further obtain
%\[\EXP{\sum_{k=0}^\infty \frac{1}{(k+1)^{1/2+\epsilon}}\left( f(\hat x_k)-f(x^*)\right)}\le
%\sum_{k=0}^\infty \frac{3}{4}\frac{K}{(k+1)^{1+\epsilon}}<\infty.\]
%Since the term under the expectation is positive, it follows that with probability~1,
%\[\sum_{k=0}^\infty \frac{1}{(k+1)^{1/2+\epsilon}}\left( f(\hat x_k)-f(x^*)\right)<\infty.\]
%Furthermore, since $0\le \epsilon\le\frac{1}{2}$, we have
%$\sum_{k=0}^\infty \frac{1}{(k+1)^{1/2+\epsilon}}=\infty$, thus implying that with probability~1,
\[\liminf_{k\to\infty}\left( f(\hat x_k)-f^*\right)=0\qquad\hbox{almost surely}.\]
Moreover, as $\{\hat x_k\}$ is bounded
and $f$ is assumed to be continuous at all points $x\in X$, it follows that one of the
(random) accumulation points
of $\{\hat x_k\}$ must be a solution of problem~\eqref{eqn:prob} almost surely.
Moreover, from this relation,
by the continuity of $f$ and the boundedness of $X$, it follows that
$\liminf_{k\to\infty} {\rm dist} (\hat x_k,X^*) =0$ almost surely.

Now, since $\hat x_k$ is a weighted average of the iterates $x_0,\ldots, x_k$
we have for all $k\ge0,$
\begin{equation}\label{eqn:last}
0\le \min_{0\le t\le k} f(x_t)-f^*\le f(\hat x_k)-f^*,\end{equation}
and by taking the expectation we obtain
\[0\le \EXP{\min_{0\le t\le k} f(x_t)-f^*}\le \EXP{f(\hat x_k)-f^*}.\]
Then, by letting $k\to\infty$ and using Theorem~~\ref{prop:bounded} we see that
\[\lim_{k\to\infty}\EXP{\min_{0\le t\le k} f(x_t)-f^*}=0.\]
Since $\min_{0\le t\le k} f(x_t)-f^*\ge0$ for all $k$, by Fatou's lemma we obtain
almost surely
\[\liminf_{k\to\infty}\left( \min_{0\le t\le k} f(x_t)-f^*\right)=0.\]
The sequence $\min_{0\le t\le k} f(x_t)$ is non-increasing and bounded below so
it has a limit, implying that
\[\lim_{k\to\infty} \min_{0\le t\le k} f(x_t)=f^*\qquad\hbox{almost surely}.\]
Again, by the continuity of $f$ and the compactness of $X$, the preceding relation implies that
almost surely
\[\lim_{k\to\infty} \left(\min_{0\le t\le k}{\rm dist} (x_t,X^*) \right)=0.\]
\end{proof}

The results of  Theorem~\ref{prop:boundedconv} are new.
As a direct consequence of Theorem~\ref{prop:boundedconv}, for an error-free subgradient
mirror-descent method~\eqref{eqn:mdstoch}, with $\a_k=\frac{a}{\sqrt{k+1 }}$, we have that
\[\lim_{k\to\infty} \min_{0\le t\le k} f(x_t)=f^*,\qquad
\lim_{k\to\infty} \left(\min_{0\le t\le k}{\rm dist} (x_t,X^*) \right)=0.\]
Thus, either there is some $k_0$ such that $x_{k_0}\in X^*$
or there is a subsequence $\{x_{k_i}\}$ converging to some optimal point.
Moreover, by relation~\eqref{eqn:last} and Corollary~\ref{cor:bounded}  we have the following error estimate
for the iterate sequence $\{x_k\}$:
\[\min_{0\le t\le k} f(x_t)-f^*\le \frac{3}{2\sqrt{k+1}}\left(\frac{d_w^2}{a}+\frac{aC^2}{2\mu_w} \right)
\qquad\hbox{for all }k\ge0.\]
For the optimal choice of the parameter $a$, the corresponding error estimate is given by
\[\min_{0\le t\le k} f(x_t)-f^*\le \frac{3}{2\sqrt{2\mu_w}}\frac{d_wC}{\sqrt{k+1}}\qquad\hbox{for all } k\ge0.\]
The preceding decrease rate for $\min_{0\le t\le k} f(x_t)-f^*$ of the order $1/\sqrt{k}$
for the error-free subgradient mirror-descent method
has been known, as shown in Theorem~4.2 of \cite{Beck2003}, where a different averaging has been used.

\section{Numerical Results}\label{sec:numerics}
In our experiment, we consider the following stochastic utility model \cite{LanNS12}:
\begin{equation}\label{eqn:f}
\min_{x\in X} f(x) ,\qquad f(x):= \EXP{\phi\left(\sum_{i=1}^n (a_i + \xi_i)x_i\right)} + \frac{\lambda}{2}\|x-z\|_2^2,
\end{equation}
where the scalars $a_i\in\re$ for $i = 1,\ldots,n$, $\lambda\ge 0$ and the vector
$z\in \mathbb{R}^n$ are given, while $\xi_i$ are independent random scalar variables,
each having the standard normal distribution, i.e.,
$\xi_i \sim {\cal N}(0,1)$ for $i=1,\ldots,n$.
The constraint set $X$ is given by
\[X := \left\{x\in \mathbb{R}^n \mid \sum_{i=1}^n x_i\le R, \, 0\le x_i\le u,\,  i = 1,\ldots,n\right\},
\]
where $R >0$ and $u >0$. The function $\phi: \mathbb{R} \to \mathbb{R}$ is piecewise linear, specifically, given by
\[
\phi(t) := \max_{j=1,\ldots,m}\left\{c_j + d_jt\right\},
\]
where $c_j$ and $d_j$ for $j = 1,\ldots,m$ are given scalars.
In the experiments, we used 10 breakpoints ($m=10$) which are all located in $[0,1]$, as shown in
Figure \ref{fig:phi}. For $\lambda >0$ in problem (\ref{eqn:f}),
the function $f(x)$ is strongly convex and, therefore, the problem has a unique optimal solution $x^*$.
With $\lambda = 0$, the problem has a nonempty optimal solution set $X^*$ due to the compactness of the constraint set $X$ and the continuity of $f$.

%For this problem, $f^* = 1$.

\begin{figure}[h]
\begin{center}
\includegraphics[scale=0.45]{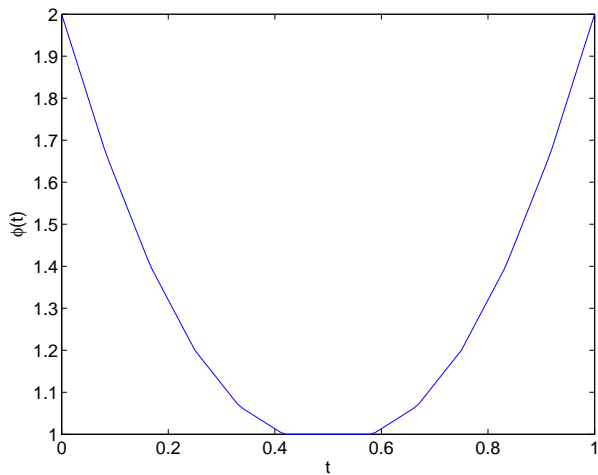}
\caption{\label{fig:phi} The utility function used in the experiments has 10 breakpoints.}
\end{center}
\end{figure}
\vskip -1pc

The experiments were conducted for two instances which have the same dimension $n=100$
and the same parameter $u=10$,
but the instances differ in the values for the parameter $R$. Thus, we have four test instances, which are labeled by \texttt{Test 1, Test 2, Test 3} and \texttt{Test 4}.
Table~\ref{tbl:instances} includes the detailed description of these instances together with
their corresponding initial points $x_0\in X$.
%Note that for all instances, we have $x_0 \in X$.
\begin{table}[t]
\begin{center}
\caption{\label{tbl:instances}The test instances and initial points}
\vspace{0.1in}
\begin{tabular}{cccc}\hline
Instance Label & $u$ & $R$ & $x_0$ \\\hline
Test 1 & 10 & 10 & $[0,\ldots,0]'$\\
Test 2 & 10 & 100 & $[0,\ldots,0]'$\\
Test 3 & 10 & 10 & $[\underbrace{1,\ldots,1}_{10},\underbrace{0,\ldots,0}_{90}]'$\\
Test 4 & 10 & 100 & $[\underbrace{10,\ldots,10}_{10},\underbrace{0,\ldots,0}_{90}]'$\\\hline
\end{tabular}
\end{center}
\end{table}

In what follows, we refer to the stochastic subgradient mirror-descent algorithm
briefly as \texttt{SSMD} algorithm. We carried out 100 independent Monte-Carlo runs to evaluate the performance of the SSMD and all the other algorithms that were used for comparison.

\subsection{Strongly Convex Objective Function}
To consider a strongly convex case, we set $\lambda = 100$ and $z=[0.5,~0,~\ldots,~0]'$ for the objective function in (\ref{eqn:f}).
We use $w(x) = \frac{1}{2}\|x\|_2^2$ for defining the Bregman distance function,
in which case the SSMD method corresponds
to the standard stochastic subgradient-projection method:
\[
x_{k+1} = P_X\left[x_k - \frac{\a_k}{\mu_k}\tilde{g}_k\right],
\]
where $P_X[x] \triangleq \arg\min_{v\in X} \|v-x\|_2^2$ is the projection of a point $x$ onto
the set $X$ and $\mu_f$ is the strong convexity parameter of the objective $f$, which is $\lambda$ here.
In the experiments, we take one sample $\xi_i$ to evaluate the stochastic gradient $\tilde{g}_k$.

For comparison, we use the accelerated stochastic approximation (\texttt{AC-SA}) algorithm by Ghadimi
et al.~\cite{GhadimiLan2012}, and we set the parameters as specified in Proposition 9 therein.
We use $x_k^{ag}$ to denote the iterates obtained by \texttt{AC-SA} algorithm,
and $\hat{x}_k$ for  a weighted-average point of the SSMD method, as defined in~\eqref{eqn:aver}.
The SSMD method is simulated for the two stepsize choices, as defined in~\eqref{eqn:step_paul}
and~\eqref{eqn:step_paul2}, which we refer to {\it step-1} and {\it step-2}, respectively.
Figure~\ref{fig:f_sc100} depicts the average (over 100 Monte-Carlo runs)
of the objective values $f(\hat{x}_k)$ (for \texttt{SSMD}) and $f(x_k^{ag})$ (for \texttt{AC-SA})
for the instances listed in Table~\ref{tbl:instances} over 100 iterations.

\begin{figure}[t]
\centering
\mbox{
\begin{subfigure}{.5\textwidth}
\centering
\includegraphics[scale=0.4]{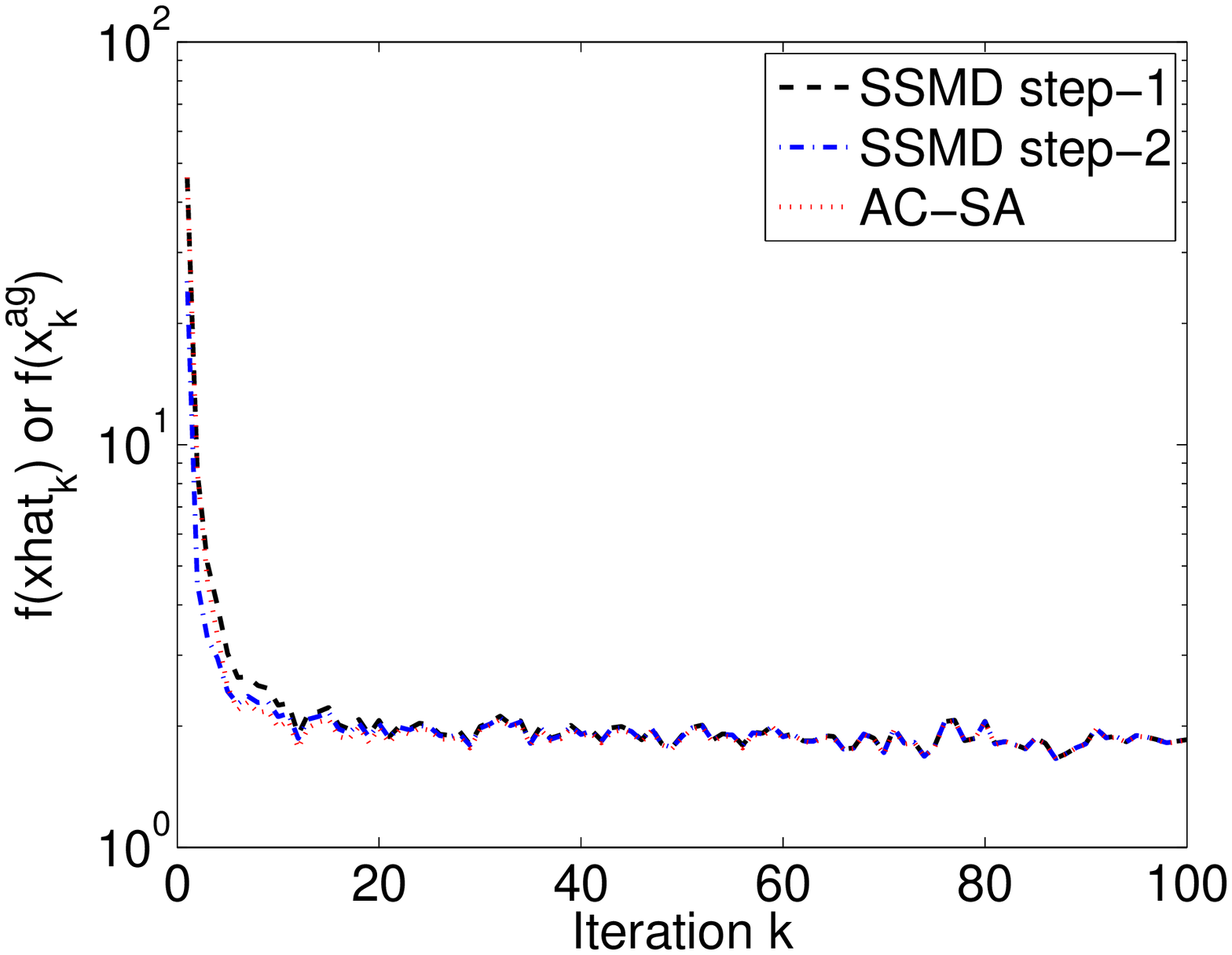}
\caption{Test 1}
\end{subfigure}
\begin{subfigure}{.5\textwidth}
\centering
\includegraphics[scale=0.4]{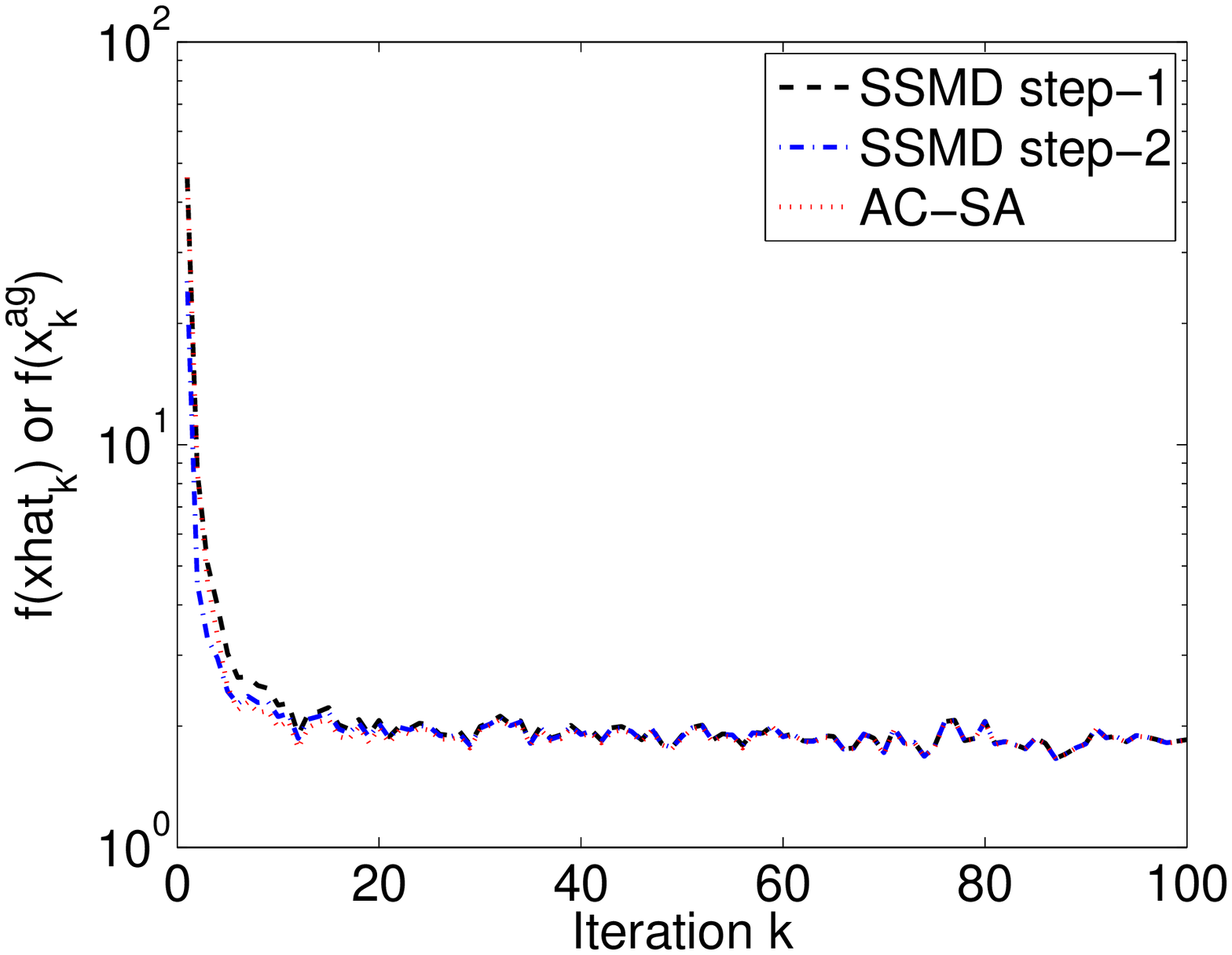}
\caption{Test 2}
\end{subfigure}
}
\mbox{
\begin{subfigure}{.5\textwidth}
\centering
\includegraphics[scale=0.4]{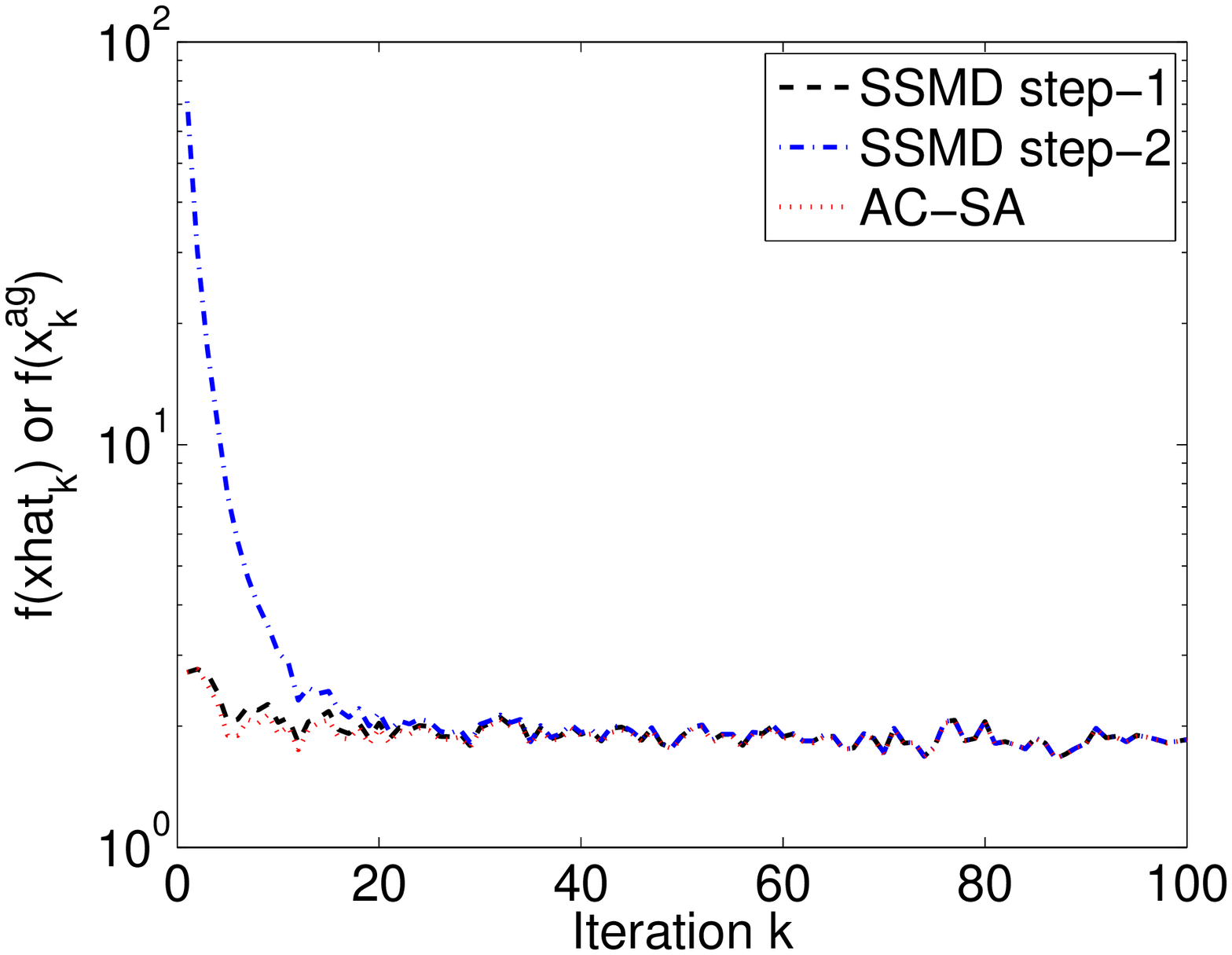}
\caption{Test 3}
\end{subfigure}
\begin{subfigure}{.5\textwidth}
\centering
\includegraphics[scale=0.4]{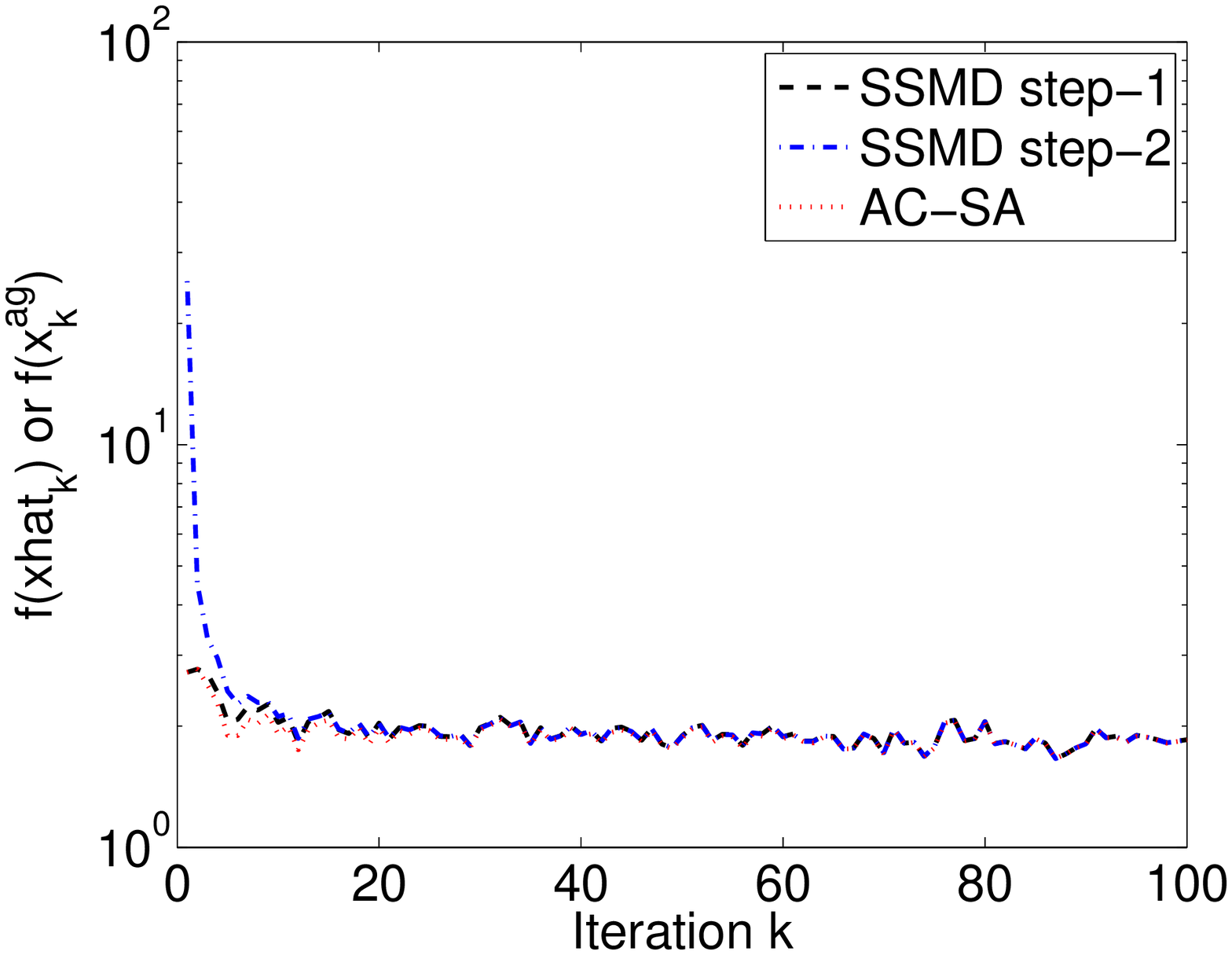}
\caption{Test 4}
\end{subfigure}
}
\caption{\label{fig:f_sc100}
\small The function values
$f(\hat{x}_k)$ and $f(x_k^{ag})$ over 100 iterations for a strongly convex $f$.}
\end{figure}

As seen from Figure~\ref{fig:f_sc100},
the algorithms show almost the same performance after about 20 iterations.
We note that the \texttt{SSMD} method with {\it step-2} (cf.~\eqref{eqn:step_paul2})
is somewhat slower within the initial 20 iterations when the initial points are not set to zero.

\subsection{Compact Constraint Set}
We set $\lambda = 0$ in the objective function given in~\eqref{eqn:f}.
With the choice $w(x)= \frac{1}{2}\|x\|_2^2$ for defining the Bregman distance function,
the \texttt{SSMD} method reduces to $
x_{k+1} = P_X\left[x_k - \a_k\tilde{g}_k\right].$
%where $P_X[x] \triangleq \arg\min_{v\in X} \|v-x\|^2$ is the projection of the point $x$ onto the set $X$.
We use the stepsize $\a_k$ as defined in~\eqref{eqn:rootstep}.

For comparison, in addition to the \texttt{AC-SA} algorithm, we also use
the Nesterov primal-dual (\texttt{PD}) subgradient method~\cite{Nesterov2009}.
For the \texttt{AC-SA} algorithm, we use the parameters specified in Proposition 8 of~\cite{GhadimiLan2012}, while in the \texttt{PD} method, we use the simple dual averaging.
In this case, the algorithms are simulated for 1000 iterations since the convergence of the methods is slower in the absence of strong convexity.

In Figure \ref{fig:f1000}, we show the performance of the algorithms in terms of the average
(over 100 Monte-Carlo runs) of the objective function.
Specifically, we plot $f(\hat{x}_k)$ for \texttt{SSMD} and \texttt{PD} (note that $\hat{x}_k$ has a different definition for \texttt{PD}) and $f(x_k^{ag})$ for \texttt{AC-SA}. The algorithms are tested for
the four instances listed in Table~\ref{tbl:instances}.
%Here, $\hat{x}_k$ is a weighted-average point of the method, defined in (\ref{eqn:aver}).

\begin{figure}[t]
\centering
\mbox{
\begin{subfigure}{.5\textwidth}
\centering
\includegraphics[scale=0.4]{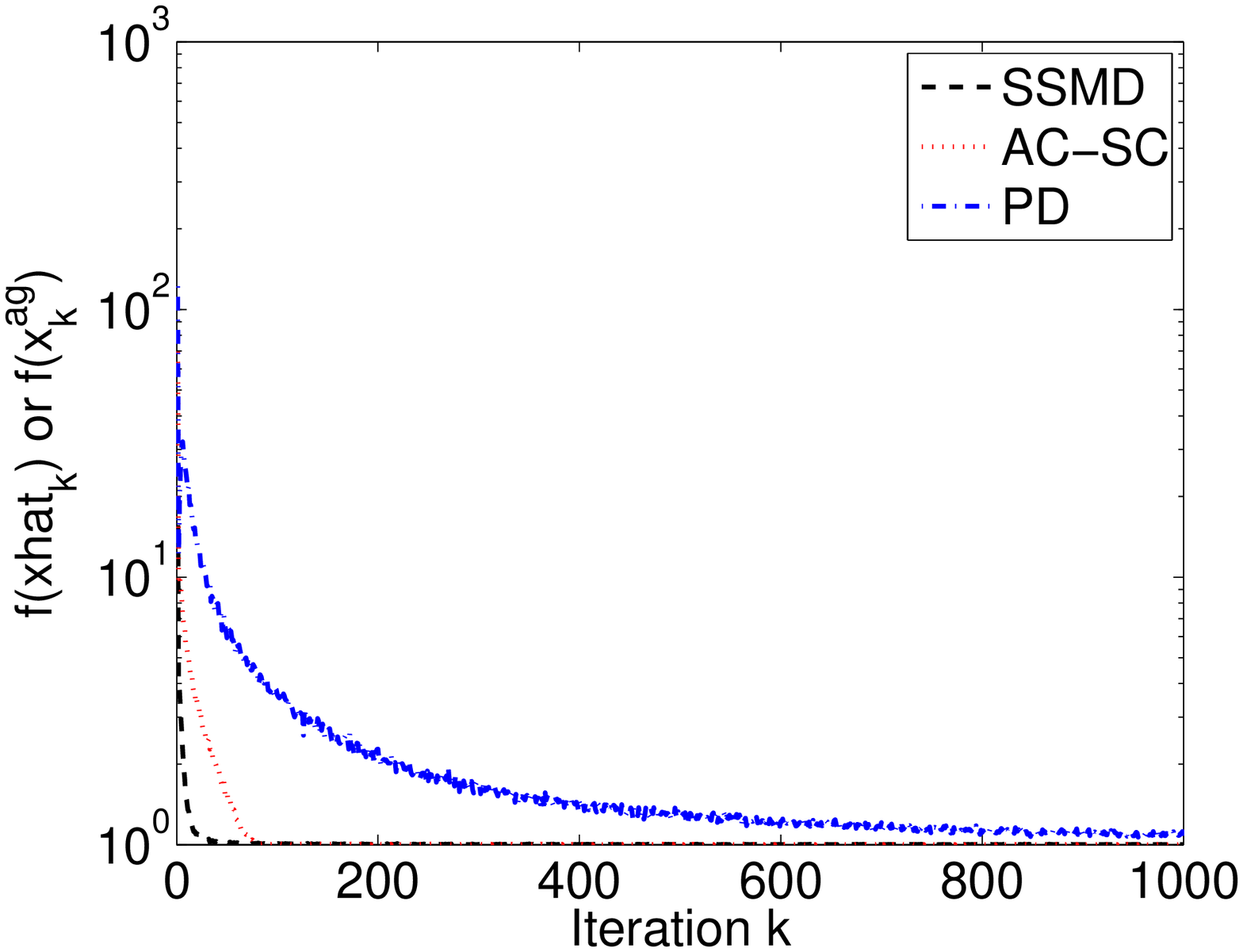}
\caption{Test 1}
\end{subfigure}
\begin{subfigure}{.5\textwidth}
\centering
\includegraphics[scale=0.4]{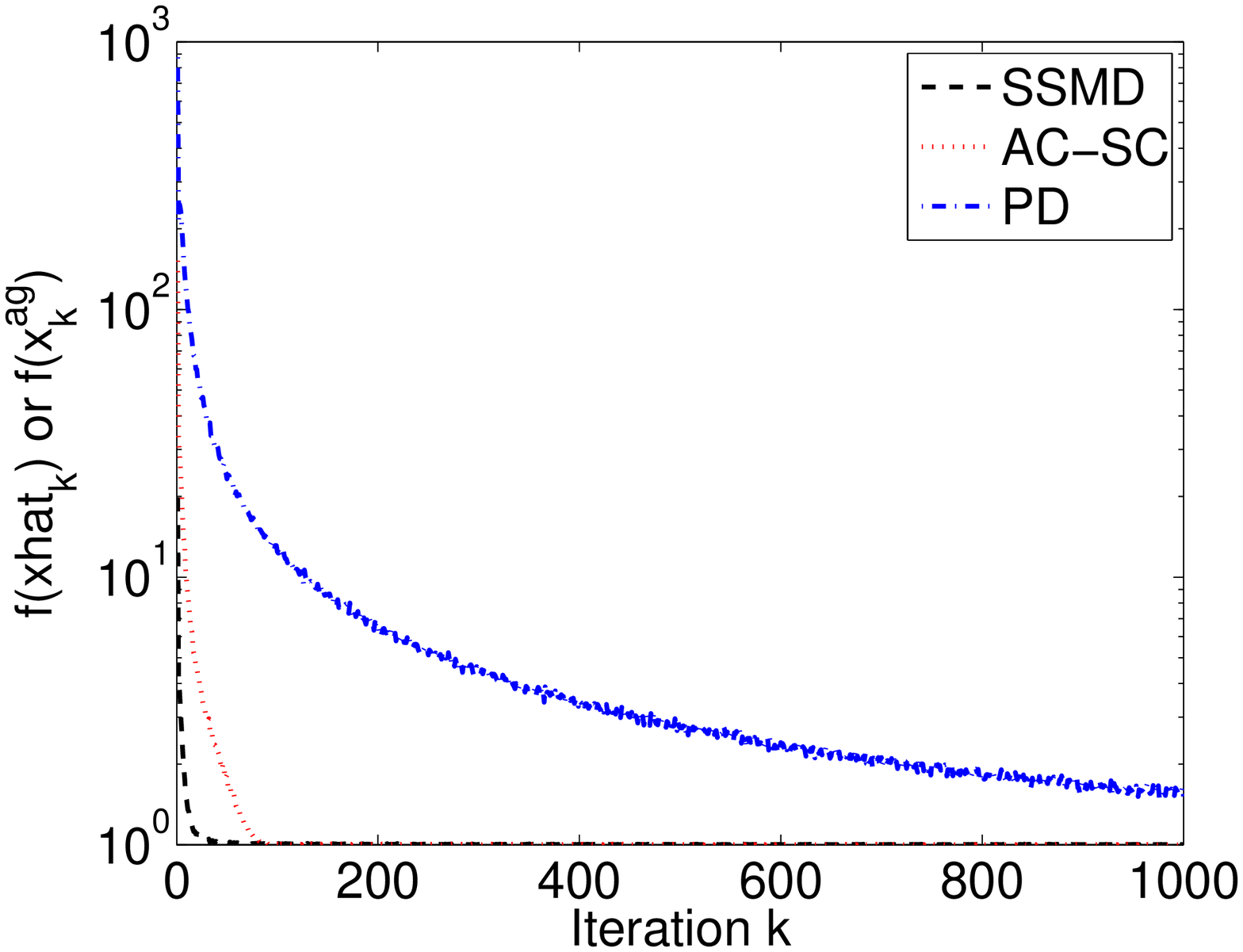}
\caption{Test 2}
\end{subfigure}
}
\mbox{
\begin{subfigure}{.5\textwidth}
\centering
\includegraphics[scale=0.4]{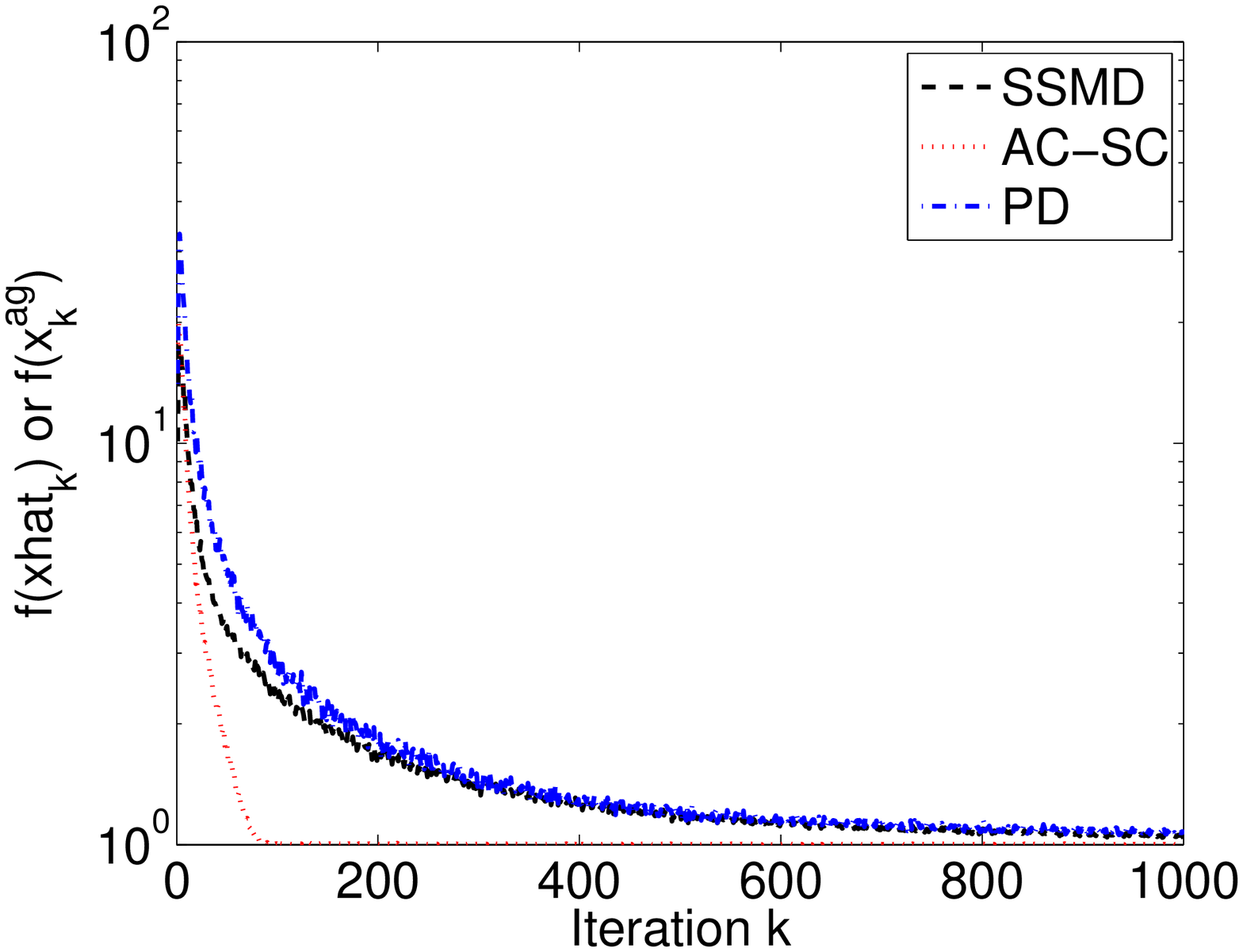}
\caption{Test 3}
\end{subfigure}
\begin{subfigure}{.5\textwidth}
\centering
\includegraphics[scale=0.4]{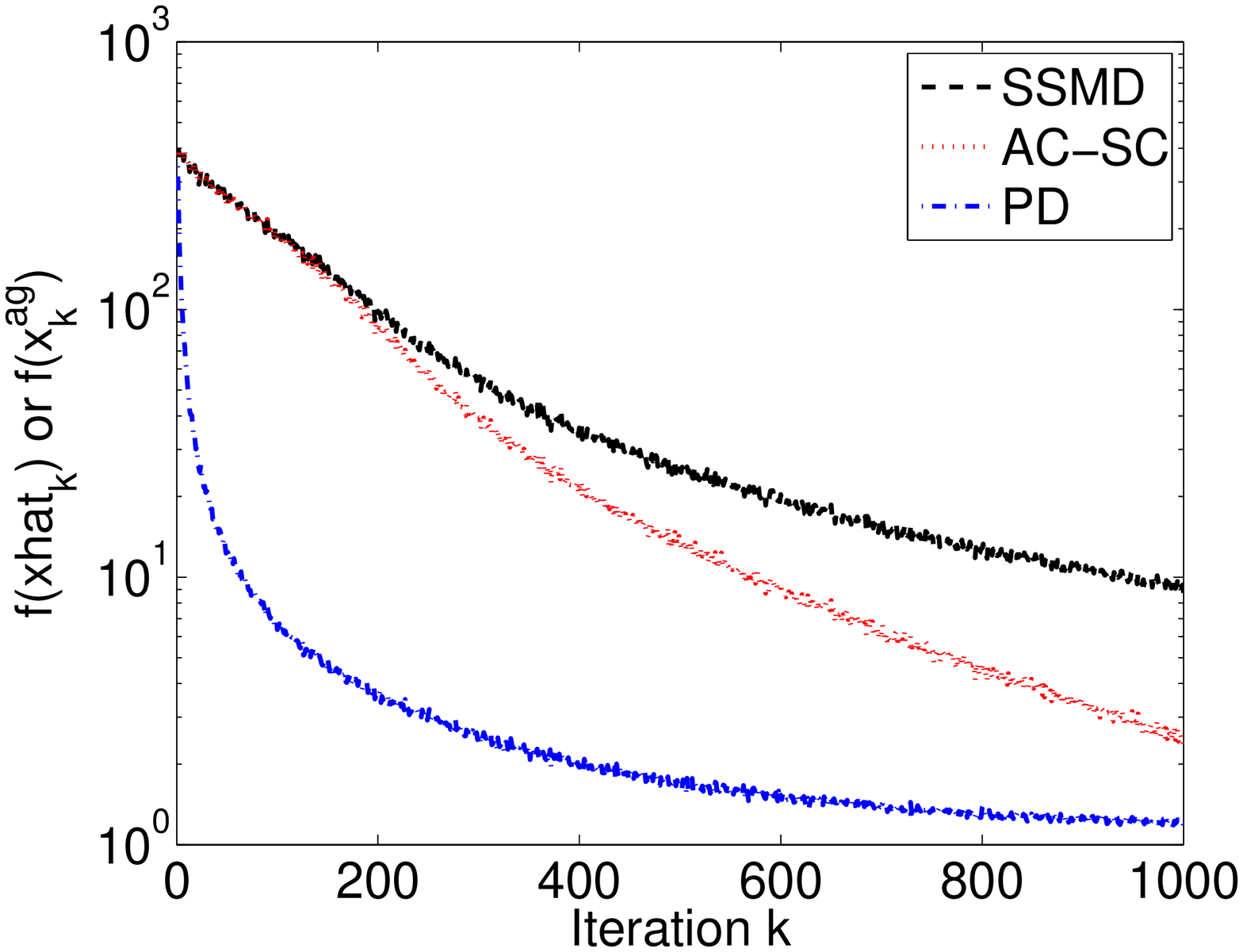}
\caption{Test 4}
\end{subfigure}
}
\caption{\label{fig:f1000} The function values
$f(\hat{x}_k)$ and $f(x_k^{ag})$ over 1000 iterations for a compact constraint set.}
\end{figure}

The \texttt{PD} method has no tunable stepsize parameters.
The \texttt{SSMD} and \texttt{AC-SA} methods have a single tunable parameter for the stepsize.
In particular, the \texttt{SSMD} method with $\a_k=\frac{1}{\sqrt{k+1}}$ has a parameter $a$,
while the \texttt{AC-SA} has a parameter $r$ with a similar role.
Both of these methods are sensitive to the choices for their respective stepsize parameters $a$ and $r$.
We tried several different choices of $a$ and $r$ in the order of tens and we plot their best results.

Overall, the performances of the three algorithms are similar and we see no reasons
to prefer one to the other.
The \texttt{SSMD} and \texttt{AC-SA} methods have a very similar behavior.
The \texttt{SSMD} algorithm performs the best for the instances Test 1 and Test 2 whose initial points
are very close the optimal set. However, in Test 4 the \texttt{AC-SA} has a better performance than the
\texttt{SSMD}.
The \texttt{PD} performs better than \texttt{SSMD} and \texttt{AC-SA} for the Test 4 instance whose feasible reagon is actually not a simplex (the inequality constraint defining the set $X$ is not active in this case).
Another interesting observation is that the \texttt{PD} method is not very sensitive to the initial points and problem instances.

\section{Conclusion}\label{sec:con}
We have considered optimality properties of the stochastic subgradient mirror-descent method
by using the weighted averages of the iterates generated by the method. The novel part of the work
is in the choice of weights that are used in the construction of the iterate averages. Through the use of
proposed weights, we can recover the best known rates for strongly convex functions and just convex functions. We also show some new convergence properties of the stochastic subgradient mirror-descent
method using the stepsize proportional to $1/\sqrt{k+1}$. In addition, we have simulation results
showing that the proposed algorithms have behavior similar to that of accelerated stochastic subgradient method~\cite{GhadimiLan2012} and the primal-dual averaging method of Nestrov~\cite{Nesterov2009}.

%%%%%%%%%%%%%%%%%%%%%%%%%%%%%%%%%%%%%%%%%%%%%%%%%%%%%%%%%%%%%%%%%

\bibliographystyle{plain}        % Include this if you use bibtex
\bibliography{stochastic_approx}           % and a bib file to produce the
                                 % bibliography (preferred). The
                                 % correct style is generated by
                                 % Elsevier at the time of printing.

\end{document}